\RequirePackage{fix-cm}
\documentclass[smallextended]{svjour3}       
\smartqed  
\usepackage{graphicx}
\usepackage{amsmath,amsfonts,amssymb}
\usepackage{dsfont}
\usepackage{array}
\usepackage{stmaryrd}
\usepackage{mathrsfs}
\usepackage{mathtools}
\usepackage{hyperref}
\usepackage{xcolor}
\usepackage[normalem]{ulem}
\hypersetup{
	colorlinks,
	linkcolor={blue!100!black},
	citecolor={blue!50!black},
	urlcolor={blue!80!black}
}
\usepackage{algorithm}
\usepackage{algorithmic}
\providecommand{\norm}[1]{\lVert#1\rVert}

\newtheorem{lem}{Lemma}
\newtheorem{Prop}{Proposition}

\usepackage{graphics}
\usepackage{graphicx}
\usepackage{epsfig}
\usepackage[left=3cm,right=3cm,top=3cm,bottom=3cm]{geometry}
\numberwithin{equation}{section}
\newcommand{\R}{\mathbb R}
\newcommand{\E}{\mathbb E}
\newcommand{\V}{\mathrm Var}
\newcommand{\C}{\mathrm Cov}

\begin{document}

\noindent  {\Large \bf On the strong uniform consistency for relative error of the regression function estimator for censoring times series model}
\vspace{1cm}
\date{}
\noindent  \textsc{\bf {BOUHADJERA Feriel. $^{1,\, 2}$}} \\
\noindent  {\it $^{1}$ Universit\'e  Badji-Mokhtar, Lab. de Probabilit\'es  et  Statistique.  BP 12, 23000 Annaba, Alg\'erie.}\\
\noindent \textsc{\bf  OULD  SA\"ID Elias.} {\sf  Corresponding author. }\\
\noindent {\it $^{2}$  Universit\'e du  Littoral C\^ote d'Opale. Lab. de Math. Pures et Appliqu\'ees. IUT de Calais. 19, rue Louis David. Calais, 62228, France.}\\
%



\vspace*{.5in}




\begin{abstract}
Consider a random vector $(X,T)$, where $X$ is $d$-dimensional and $T$ is one-dimensional. We suppose that the random variable $T$ is subject to random right censoring and satisfies the $\alpha$-mixing property. The aim of this paper is to study the behavior of the kernel estimator of the relative error regression and to establish its uniform almost sure consistency with rate. Furthermore, we have highlighted the covariance term which measures the dependency. The simulation study shows that the proposed estimator performs well for a finite sample size in different cases.
\vspace*{.2in}

\end{abstract}
\vspace*{.3in}

\section{Introduction}
\label{intro}
Let $T$ be a strictly positive random variable (r.v.) representing the survival time of an individual taking part to an experimental study and let $X$  be a vector of covariate taking values in $\R^d$ that gives us information about the individuals (age, sex, \dots). This paper is concerned with the nonparametric regression model: $T=m(X)+\varepsilon,$ where $m$ is a regression function and $\varepsilon$ is a r.v. (corresponding to the residual) such that $\E[\varepsilon|X]=0$. Recall that, $m(\cdot)$ is usually modeled by the following minimization  problem 
 $\E\big[\big (T-m(X)\big)^2\big| X\big]$. However this loss function is unsuitable when the data contains some outliers, which is a relatively common case in practice. To avoid this drawback, another approach is to build an efficient estimator of $m(\cdot)$ given by the minimization of the mean squared relative error  given by
\begin{equation}\label{MSRE}
\displaystyle \E\left[ \Big(\frac{T-m(X)}{T} \Big)^2\Big|X\right], \;\;\;\text{for} \; T>0
\end{equation}
This kind of model  is called relative error regression (RER)  which has been studied by many authors. We can refer to \hyperref[Narula]{Narula and Wellington (1977)}, \hyperref[Makridakis]{Makridakis {\it et al.} (1984)} in the parametric case. Recall that  \hyperref[Park]{Park and Stefanski (1998)} showed that the solution of (\ref{MSRE}) satisfies, for $x \in \R^d$,  
\begin{equation}
\displaystyle \frac{\E \big[T^{-1}|X=x\big]}{\E \big[T^{-2}|X=x\big]} \leq \E \big[T|X=x\big].
\end{equation}
In the nonparametric analysis, there exist some papers dealing with the estimation of the RER. Without pretending to exhaustivity, we quote
\hyperref[Jones]{Jones {\it et al.} (2008)}   considered kernel and local linear approach to estimate the regression function, \hyperref[Attouch]{Attouch {\it et al.} (2015)} regarded the problem of estimating the regression function for spatial data and \hyperref[Demongeot]{Demongeot {\it et al.} (2016)} considered the case where the explanatory variable are of functional type of data. \hyperref[Shen]{Shen and Xie (2013)} obtained the strong consistency of the regression estimator under $\alpha$-mixing data. Their result has  been generalized to dependent case by \hyperref[Li]{Li {\it et al.} (2016)}.\\ 
In the case of incomplete data, for independent and identically distributed (i.i.d.) random variables under random right censoring, \hyperref[Guessoum1]{Guessoum and Ould Said (2008)} studied the consistency and asymptotic normality of the kernel estimator of the regression function.  \\ 
Many statistical results have been established by considering independent samples. 
It is then interesting to consider the more realistic situation when the observation are no longer i.i.d. This is for exemple the case, in clinical trials studies, not infrequently, patients from the same hospital have correlated survival times, due  to unmeasured variables such as the skill or training of the staff or the quality of the hospital equipment (for more details, see : \hyperref[lipsitz]{Lipsitz and Ibrahim, 2000}).  \\
Few papers deal with the regression function under censoring in the dependent case. We can cite \hyperref[Cai1]{Cai (1998)} who studied the asymptotic properties of Kaplan-Meier's estimator of censored dependent data and \hyperref[Cai2]{Cai (2001)} who addressed the estimation of the distribution function for censored time series data. \hyperref[Ghouch]{El Ghouch and Van Keilegom (2008, 2009)} estimated the regression and conditional quantile  functions by applying local linear method. \hyperref[Guessoum]{Guessoum and Ould Said (2010, 2012)} studied the consistency and the asymptotic normality of the kernel estimator for the regression function for censored dependent data. \\ 
Here we derive the uniform consistency result over a compact set with rate  of the RER estimator for dependent case and censored data by highlighting the  covariance term which does not appear in many papers. This paper is organized as follows. In  \hyperref[Sec2]{Section\,2}, we recall some notations and definitions needed in our model. The hypotheses and main result are given in \hyperref[Sec3]{Section\,3}. Simulations study are drawn in \hyperref[Sect4]{Section\,4}. Finally, the proofs are relegated to \hyperref[Sect5]{Section\,5} with some auxiliary results.
\section{Presentation of the model}\label{Sec2}
Consider a randomly right-censored  model given by two non-negative stationary sequences,  $(T_i)_{1 \leq i \leq n}$ which represents the survival time with common unknown absolutely continuous distribution function (d.f.) $F$ and $(C_i)_{1 \leq i \leq n}$ the censoring time with common unknown d.f. $G$. In this context, we observe the pairs $(Y_i,\delta_i)$, where 
\[Y_i=T_i \wedge C_i ,\;\;\;\; \;\;\;\; \delta_i=\mathds{1}_{\{T_i \leq C_i\}},\;\; i=1,\dots,n\]
and  $\mathds{1}_E$ denotes the indicator function of the set $E$.  Let $(X_i)_{1 \leq i \leq n} $ be a sequence of copies of the random vector $X \in \R^d$ and denote by $X_{1},\dots,X_{d}$ the components of $X$.  The study we perform below is then on the set of observations $(Y_i,\delta_i,X_i)_{1 \leq i \leq n}$.
Having in mind this kind of model, we define a pseudo-estimator of the relative error for the regression function, for all $x \in \R^d$, by
\begin{equation}\label{pseudo}
\tilde{m}(x)=\frac{\displaystyle{\sum_{i=1}^{n}\frac{\delta_{i} Y_{i}^{-1}}{\overline{G}(Y_i)}K_d(x-X_i)}}{\displaystyle{\sum_{i=1}^{n}\frac{\delta_{i} Y_{i}^{-2}}{\overline{G}(Y_i)}K_d(x-X_i)}}=:\frac{\tilde{r}_{1}(x)}{\tilde{r}_{2}(x)},  
\end{equation}
with $\displaystyle \tilde{r}_{\ell}(\cdot)=\frac{\tilde{m}_{\ell}(\cdot)}{f(\cdot)}$ and $\tilde{m}_{\ell}(\cdot)=\int_{\R^d} y^{-\ell}f(y|x)dy$, where $K_d(\cdot)=K_d( \cdot/h_n)$ is a density function defined on $\R^d$ and $h_n$ is a sequence of positive numbers. \\
In this kind of model, it is well known that the empirical distribution is not a consistent estimator for the df $G$. Therefore, \hyperref[Kap-Mei]{Kaplan-Meier (1958)} proposed a consistent estimator for survival function $\overline{G}(\cdot)=1-G(\cdot)$ which is defined as 
\begin{equation}\label{K_M}
\overline{G}_n(t)=
\left\{
\begin{array}{cl}
\displaystyle{\prod_{i=1}^{n}{\left(1-\frac{1-\delta_{(i)}}{n-i+1}\right)}^{\mathds{1}_{\{Y_{(i)}\leq t\}}}} & \quad \text{if} \quad t<Y_{(n)} \\
0 & \quad \text{otherwise},
\end{array}
\right.
\end{equation}
where $Y_{(1)}\leq Y_{(2)}\leq \dots \leq Y_{(n)}$ are the order statistics of the $Y_{i}$'s and $\delta_{i}$ is the indicator of non-censoring. The properties of the K-M estimator for dependent variables can be found in \hyperref[Cai1]{Cai (1998, 2001)}. Then a calculable estimator of $m(\cdot)$ is given by 
\begin{equation*}\label{estimRRC}
\widehat{m}(x)=\frac{\widehat{m}_1(x)}{\widehat{m}_2(x)}
\end{equation*}
where for $\ell=1,\; 2$
\begin{equation*}\label{estimerell}
\displaystyle \widehat{m}_{\ell}(x)=\frac{\widehat{r}_{\ell}(x)}{\widehat{f}_X(x)}=\frac{\displaystyle{\sum_{i=1}^{n} \frac{\delta_i Y_{i}^{-\ell}}{\overline{G}_n (Y_i)}K_d(x-X_i)}}{\displaystyle{\sum_{i=1}^{n} K_d\left(x-X_i\right)}}
\end{equation*}
and $\widehat{f}_X(\cdot)$ is the kernel estimator of the marginal density function $f_X(\cdot)$.\\\\
In  order to define the  $\alpha$-mixing notion,  we  will use the following notations. Denote by $\displaystyle  \mathcal{F}_{1}^{k}(Z)$ the $\sigma-$algebra generated by $\{Z_j,\; 1\leq j\leq k\}$.
\begin{description}
	\item[Definition] \textit{Let }$\left\{ Z_{i},i=1,2,...\right\} $\textit{\
		denote a sequence of rv's. Given a positive integer n, set}%
	\begin{equation*}
	\alpha (n)={\sup }\left\{ \left\vert \mathbb{P(}A\cap B)-\mathbb{P(}A)%
	\mathbb{P}(B)\right\vert :A\in \mathcal{F}_{1}^{k}(Z)\text{ and }B\in 
	\mathcal{F}_{k+n}^{\infty }(Z),\; k\in {\mathrm{I\!N}}^*\right\}.
	\end{equation*}
\end{description}

\noindent \textit{\ The sequence is said to be }$\alpha$-\textit{mixing
	(strong mixing) if the mixing coefficient }$\alpha (n)\rightarrow 0$\textit{%
	\ as } $n\rightarrow \infty .$

\noindent Many processes  fulfill the strong mixing property. We quote
here, the usual ARMA processes which are geometrically strongly mixing, 
\textit{i.e.,} there exist $\rho \in \;(0,1)$ and $a>0$ such that, for any $%
n\geq 1$, $\alpha (n)\leq a\rho ^{n}$ (see, \textit{e.g.},  \hyperref[Jones]{Jones (1978)}.
The threshold models, the EXPAR models (see,\hyperref[Ozaki]{ Ozaki (1979)}, the simple ARCH
models (see \hyperref[Engle]{Engle (1982)}, their GARCH extension (see \hyperref[Bollerslev]{Bollerslev (1986)} and
the bilinear Markovian models are geometrically strongly mixing under some
general ergodicity conditions.\newline
\noindent We suppose that the sequences $\left\{ T_{i},i\geq 1\right\} $ and 
$\left\{ C_{i},i\geq 1\right\} $ are $\alpha $-mixing with coefficients $%
\alpha _{1}(n)$ and $\alpha _{2}(n)$, respectively. \hyperref[Cai2]{Cai (2001)} Lemma 2
showed that $\left\{ Y_{i},i\geq 1\right\} $ is then strongly mixing, with
coefficient%
\begin{equation*}
\alpha (n)=4\ \max (\alpha _{1}(n),\alpha _{2}(n)).
\end{equation*}%
\noindent From now on, we suppose that $\left\{ \left( Y_{i},\delta
_{i},X_{i}\right) \ \ i=1,...,n\right\} $ is strongly mixing with mixing's
coefficient $\alpha (n)$ such that $\alpha (n)=O(n^{-\nu })$ for some $\nu >3$.

\section{Hypotheses and main results}\label{Sec3}
Let $\mathcal{C}$  be a compact set in $\R^d$. We define the endpoint of $F$ by $\tau_F=sup\{x,\bar{F}(x)>0\}$ and we assume that $\tau_F<\infty$ and $\bar{G}(\tau_F)>0$.\\
All along the paper, we denote by $r_{\ell}(x)= \displaystyle{\int_{\R^d} t^{-\ell} f_{X,T}(x,t)dt}$
where $f_{T,X}(\cdot,\cdot)$ is the joint density of $(T,X)$. For any generic strictly positive constant $M$, we assume
\begin{equation} \label{majorant}
\forall T>0, \;\;  \exists M,  \;\;\text{such that}  \;\; M \geq T^{-\ell},  \;\;\text{for}  \;\; \ell=1,2.
\end{equation}
In order to present our result, we have to introduce the following  notations and hypotheses.
\begin{itemize}
	\label{H1}\item[\textbf{H1.}] The bandwidth $h_n$ satisfy $\displaystyle{\lim_{n \rightarrow +\infty} \frac{nh_n^d}{\log n}= +\infty},$ 
	\label{H2}\item[\textbf{H2.}] $\displaystyle{\lim_{n \rightarrow +\infty}\frac{\log n}{nh_n^{2d/v}}= 0},$ 
	\label{H3}\item[\textbf{H3.}] $\exists \psi>0$, $\exists c>0$, such that 
	\[\displaystyle{cn^{\frac{\gamma(3-\nu)}{\gamma(\nu+1)+2\gamma+1}+\psi d} \leq h_n^d}, \;\;\; \text{for all} \;\;\; \nu>3. \]
	\label{K1}\item [\textbf{K1.}]  The kernel $K_d$ is bounded and  \[\forall(z_1,z_2)\in \mathcal{C}^2,  |K_d(z_1)-K_d(z_2)|\leq \norm{z_1-z_2}^{\gamma} \;\;\;\text{for} \;\;\;\gamma>0,\]
	\label{K2}\item [\textbf{K2.}] $\displaystyle \int_{\R^d} \norm{t} K_d(t)dt < +\infty,$ with $\displaystyle \norm{t}=\sum_{i=1}^{n}|t_i|,$
	\label{K3}\item [\textbf{K3.}] $\displaystyle \int_{\R^d} \norm{t} K_d^2(t)dt < +\infty$ and $\displaystyle \int_{\R^d}  K_d^2(t)dt < +\infty.$
	\label{D1}\item [\textbf{D1.}] The function $ \displaystyle r_{\ell}(x)$ defined in (\ref{pseudo}) is continuously differentiable and $\displaystyle \sup_{x \in \mathcal{C}} \left| \frac{\partial r_{\ell}(x)}{\partial x_i}\right|<+\infty$ for $i=1,\dots,d,$
	\label{D2}\item [\textbf{D2.}] The function $\displaystyle \theta_{\ell}(x):=\int_{\R^d} \frac{t^{-\ell}}{\bar{G}(t)}f_{X,T}(x,t)dt$ is continuously differentiable and $\displaystyle \sup_{x \in \mathcal{C}} \left| \frac{\partial \theta_{\ell}(x)}{\partial x_i}\right|<+\infty$ for $i=1,\dots,d,$ and $\ell=2,3,4$.
	\label{D3}\item [\textbf{D3.}] The joint density $ \displaystyle f_{i,j}(\cdot,\cdot)$ of $(X_i,X_j)$ exists and satisfies for $ \displaystyle \ell=1,2$ \[\displaystyle{ \sup_{i,j} \sup_{u,v \in \mathcal{C}} | f_{\ell,i,j}(u,v)- f_{\ell,i}(u) f_{\ell,j}(v) | \leq C < \infty,}\] where $M$ is a positive constant.
\end{itemize}
\subsection{Some comments on the hypotheses} 
Hypotheses \hyperref[H1]{\textbf{H1}} and \hyperref[H2]{\textbf{H2}} are very common in both independent and dependent cases.  Furthermore, \hyperref[H3]{\textbf{H3}} permits to estimate the covariance term. The  hypotheses concerning the kernel \hyperref[K1]{\textbf{K}} are technicals and it is well-known that  it does not improve the quality of the estimation.  The  \hyperref[D1]{\textbf{D1}}  intervenes in Lemma 1, however  hypotheses  \hyperref[D2]{\textbf{D2}} and \hyperref[D3]{\textbf{D3}}  intervene in \hyperref[lem4]{Lemma 4} to deal with the covariance term.
\subsection{Bandwidth selection} \label{discussion-bandwidht}
\noindent Note that: It is well-known that the choice of the kernel does not affect the quality of the estimation. In contrast, the bandwidth parameter $h_n$  has a great  influence on the quality of  the estimator. A parameter that is too small causes the appearance of artificial details in the graph of the estimator, and  for a large enough value of the bandwidth $h_n$, the majority of the features is on the contrary erased. The choice of the bandwidth $h_n$ is therefore a central question  in nonparametric estimation. Recall that in the literature, there are mainly three methods, the "rule of thumb", "plug-in" and "cross-validation".  Each method has its merits and drawbacks.  We point out that the latter is very popular and its main idea is to minimize the following criterion \[\displaystyle CV_{h_n}=\frac{1}{n-1}\sum_{i=1}^n\big(Y_i-\widehat{m}_{-i,\, h_n}(X_i)\big)^2\] where  $\widehat{m}_{-i,\, h_n}(X_i)$ is the  estimator of $m(\cdot)$ obtained by raising the observation $(X_i, Y_i)$ in the sense of practical point of view. Even if the latter has the drawbacks that  it is  very variable and can give an underestimation of $h_{opt}$, it remains the most common used method. In our entire simulation study, we adopt the cross-validation method (see: \hyperref[section4]{Sect. 4}).
The following theorem establishes the almost sure uniform convergence of $\widehat{m}$ towards $m$.
\begin{theorem}\label{theorem}
	Under hypotheses \hyperref[H1]{\textbf{H1-3}}, \hyperref[D3]{\textbf{D1-3}}, we have, for $\ell=1,2$,
	\begin{equation*}
	\sup_{x \in \mathcal{C}}| \widehat{m}(x)- m(x)| =O_{a.s.}\left(\sqrt{\frac{\log n}{nh_n^d}}+\sqrt{\frac{\log n}{n h_n^{2d/v}}}\right)+O(h_n)  \quad \text{as} \quad n \longrightarrow \infty.
	\end{equation*}
\end{theorem}
	\begin{remark}
	We point out that in our result we highlight the covariance term which give us how the dependency intervenes. This point is rarely given in the dependent case of many papers. In the latter the authors made an additional hypotheses to vanish the covariance term to get an analogous result as in the independent case.
\end{remark}
\section{Simulation study}\label{Sect4}
The aim of this part is to examine the performance of our estimator $\widehat{m}(x)$ by considering some fixed size particular cases. We do it by varying the dependency rate and the censoring percentage (C.P.). We compare the efficiency of the implemented method to the classical regression (CR) estimator  defined in \hyperref[Guessoum]{Guessoum and Ould Sa\"id (2010)}.\\
In the next paragraph, we recall a result of {\hyperref[port]{Port (1994)}} which permits to calculate the theoretical regression function that will be used throughout this section (see formula (\ref{theo_func}) below).
\begin{Prop}\label{prop_port}
	Let $q_1(X)$ and $q_2(X)$ be two random variables with means: $\mu_1$ and $\mu_2$ and variances: $v_1$ and $v_2$ respectively, and covariance $v_{12}$. Let $(X_i)_{1 \leq i \leq n}$ be an i.i.d. sequence of r.v. and defined by 
	\[\widehat{\Sigma}_1= \frac{1}{n} \sum_{1}^{n} q_1(X_i) \;\; \text{and} \;\; \widehat{\Sigma}_2= \frac{1}{n} \sum_{1}^{n} q_2(X_i)\]
	and $\widehat{R}=\displaystyle{\frac{\widehat{\Sigma}_1}{\widehat{\Sigma}_2}}$ then the second order approximation of $\E[\widehat{R}]$ is 
	\[\E[\widehat{R}] \approx \frac{\mu_1}{\mu_2}+ \frac{1}{n} \left( \frac{\mu_1 v_2}{\mu_2^3}-\frac{v_{12}}{\mu_2^2}\right).\]
\end{Prop}
\begin{algorithm}
	\caption{}
	\begin{algorithmic} 
		\REQUIRE $0< \rho <1$, $X_0 \leadsto \mathcal{N}(1,0.1)$ , $\varepsilon \leadsto \mathcal{N}(0,1)$, $C \leadsto \mathcal{N}(3+a,1)$ with $a$ being a parameter that adapts the censorship percentage C.P. 
		\begin{description}
			\item[Step 1 :]	We consider the strong mixing two-dimensional process generated by
			\begin{equation*}
			\left\{
			\begin{aligned}
			X_i&=c+\rho X_{i-1}+\sqrt{1-\rho^2}\varepsilon_i,\\
			T_i&=X_{i+1}, \;\; i=1,\dots,n
			\end{aligned}
			\right.
			\end{equation*}
			\item[Step 2 :] Given $X_1=x$, we have $T_1=c+\rho x+\sqrt{1-\rho^2}\varepsilon_1$. Using Port property (see \hyperref[prop_port]{proposition 1}) the theoretical function becomes
			\begin{equation} 
			\label{theo_func} m(x)=\frac{\E \big[T^{-1}|X=x\big]}{\E \big[T^{-2}|X=x\big]}=c+\rho x+\frac{1-\rho^2}{c+\rho x}.
			\end{equation}
			\item[Step 3 :] Determine $Y_i=T_i\wedge C_i$ and $\delta_i={\mathds{1}}_{\{T_i \leq C_i\}}$ which gives the observed sample  $\{(X_i,Y_i,\delta_i), 1 \leq i \leq n \}$.
			\item[Step 4 :] The K-M estimator of $\bar{G}(\cdot)$ is calculated from (\ref{K_M}).
			\item[Step 5 :] The Gaussian kernel $\displaystyle K(z)=\frac{1}{\sqrt{2\pi}}\exp\left(\frac{-z}{2}\right)$ is used as kernel function for the estimator and we choose the optimal bandwidth $h_{opt}$ by the cross validation method (see \hyperref[discussion-bandwidht]{Subsection 3.2}) from $[0.01,2]$ by step of $0.01$ and satisfying \hyperref[H3]{\textbf{H3}}.
		\end{description}
	\end{algorithmic}
	{\bf Output:} Calculate the RER estimator given by (\ref{estimerell}) for $x\in[1,4]$ and $h_{opt}$.
\end{algorithm}
\subsection{Linear case }
In this subsection, we observe the finite sample performance of our estimator (RER) for weak and strong dependency when the theoretical function is of linear form.
\subsubsection{Weak dependency}

\paragraph{$\bullet$ Effect of sample size:} It is easy to see from \hyperref[figure1]{Figure 1} that the quality of fit is better when $n$ increases for a fixed C.P. and $\rho$. 
\begin{figure}[!h]
	\begin{minipage}[c]{.26\linewidth}
		\includegraphics[width=1.3\textwidth]{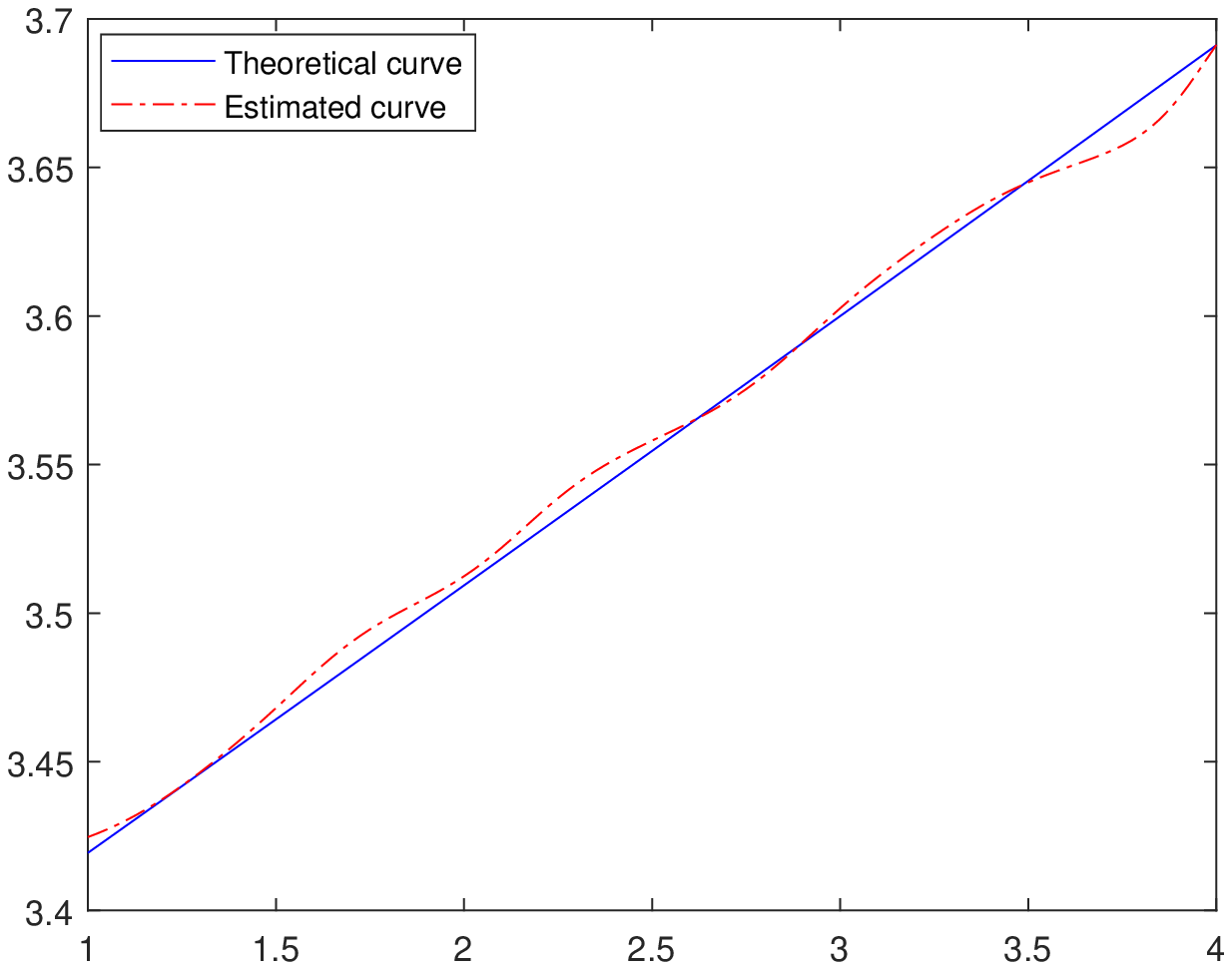}   
	\end{minipage} \hfill
	\begin{minipage}[c]{.26\linewidth}
		\includegraphics[width=1.3\textwidth]{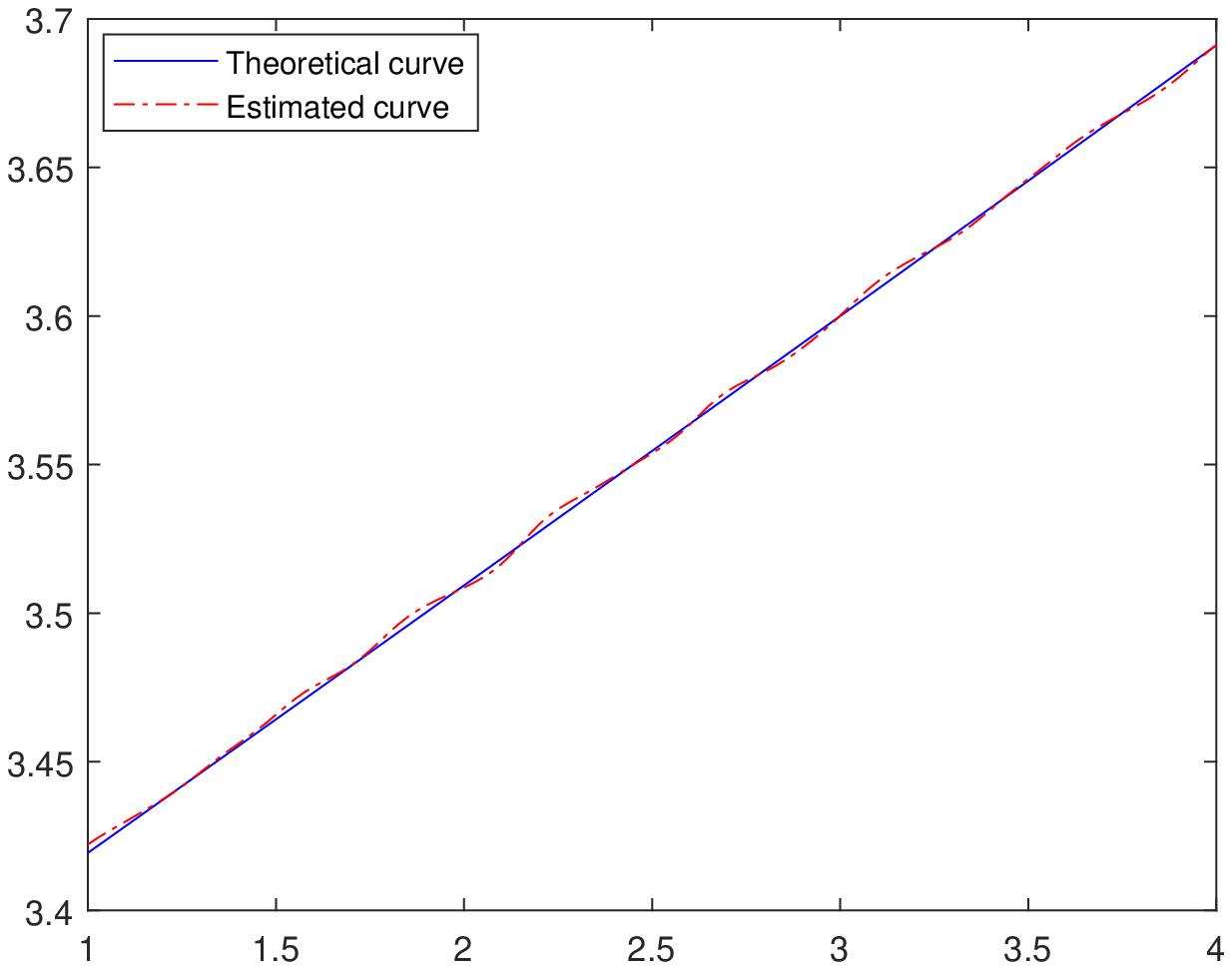}
	\end{minipage} \hfill
	\begin{minipage}[c]{.26\linewidth}
		\includegraphics[width=1.3\textwidth]{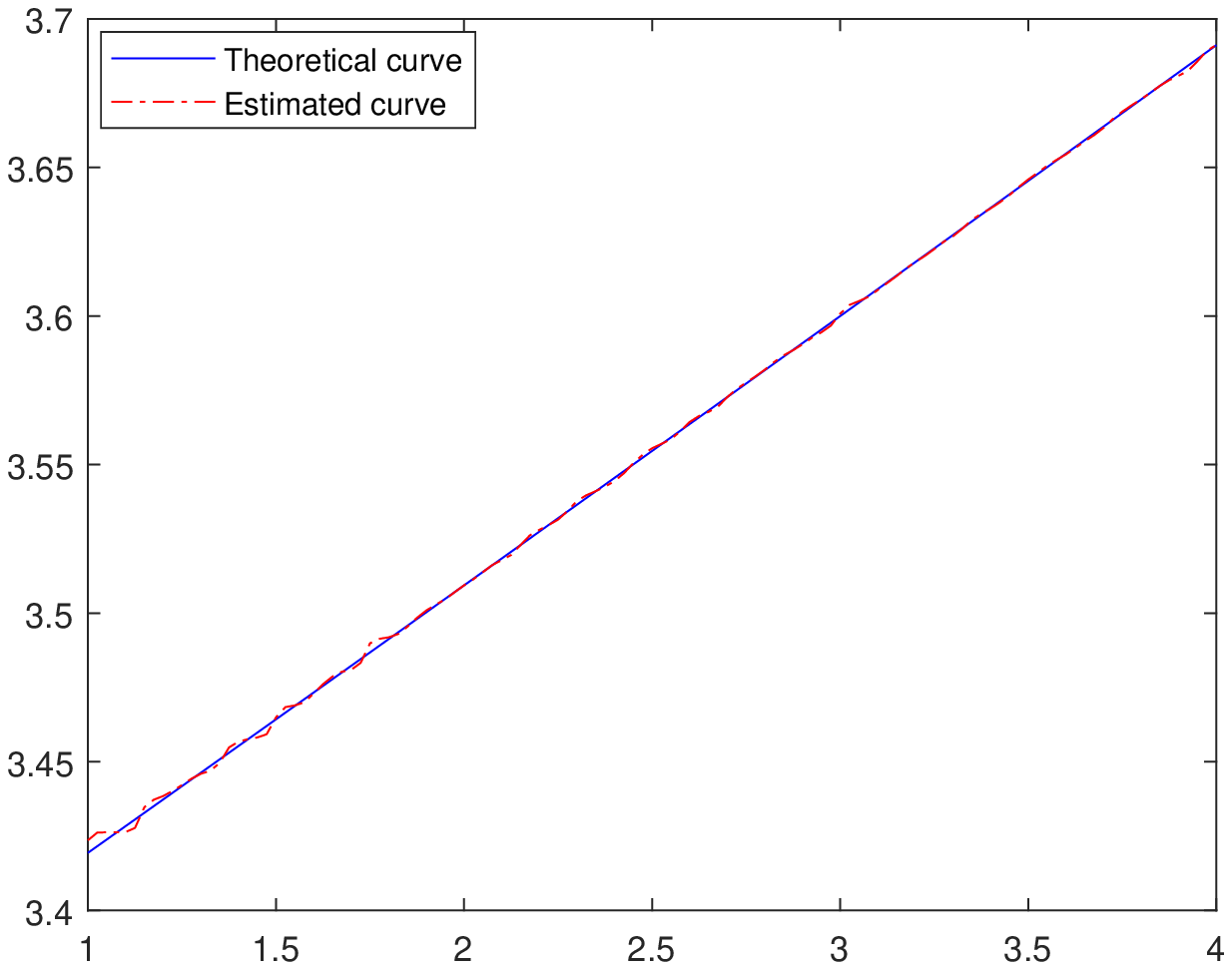}
	\end{minipage}\hfill\hfill
	\caption{ \textcolor{blue}{$m(x)$}, \textcolor{red}{$\widehat{m}(x)$} with $c=3$, $\rho=0.1$ and C.P. $\approx 40\%\;$ for $n=100, 300 \;\text{and}\; 500$ respectively.}\label{figure1}
\end{figure}	
	
\paragraph{$\bullet$ Effect of C.P.:} To visualize the global performance of the RER estimator under censorship, we set $\rho=0.1$ and vary the C.P. In this case, there is 
 more variation in the resulting estimator, but generally remains close to the theoretical curve even for a high C.P. (\hyperref[figure2]{Figure 2} displays the results). In conclusion, our estimator is still resistant to the effect of censorship when dependency is weak. 
\begin{figure}[!h]
	\begin{minipage}[c]{.26\linewidth}
		\includegraphics[width=1.3\textwidth]{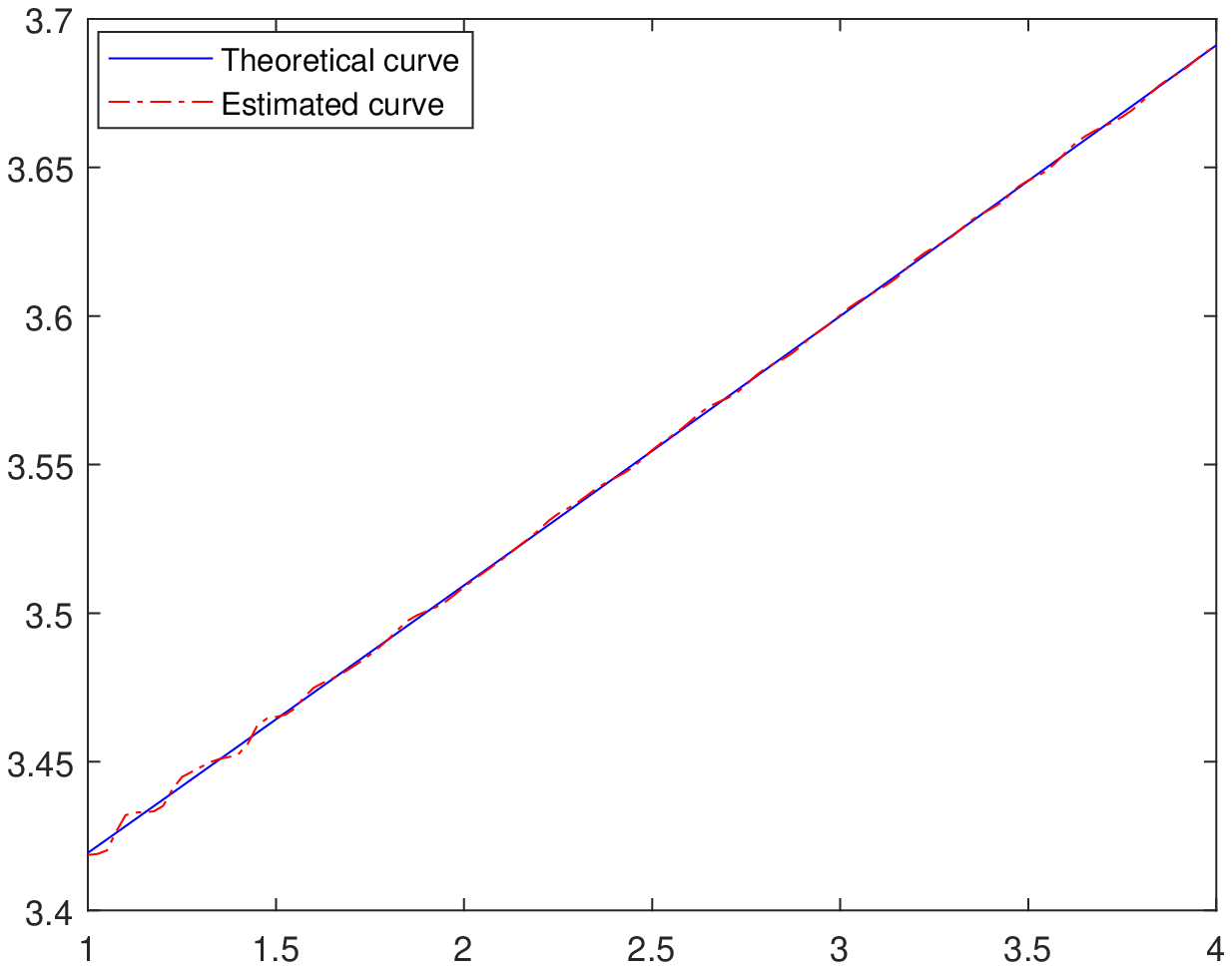}  
	\end{minipage} \hfill
	\begin{minipage}[c]{.26\linewidth}
		\includegraphics[width=1.3\textwidth]{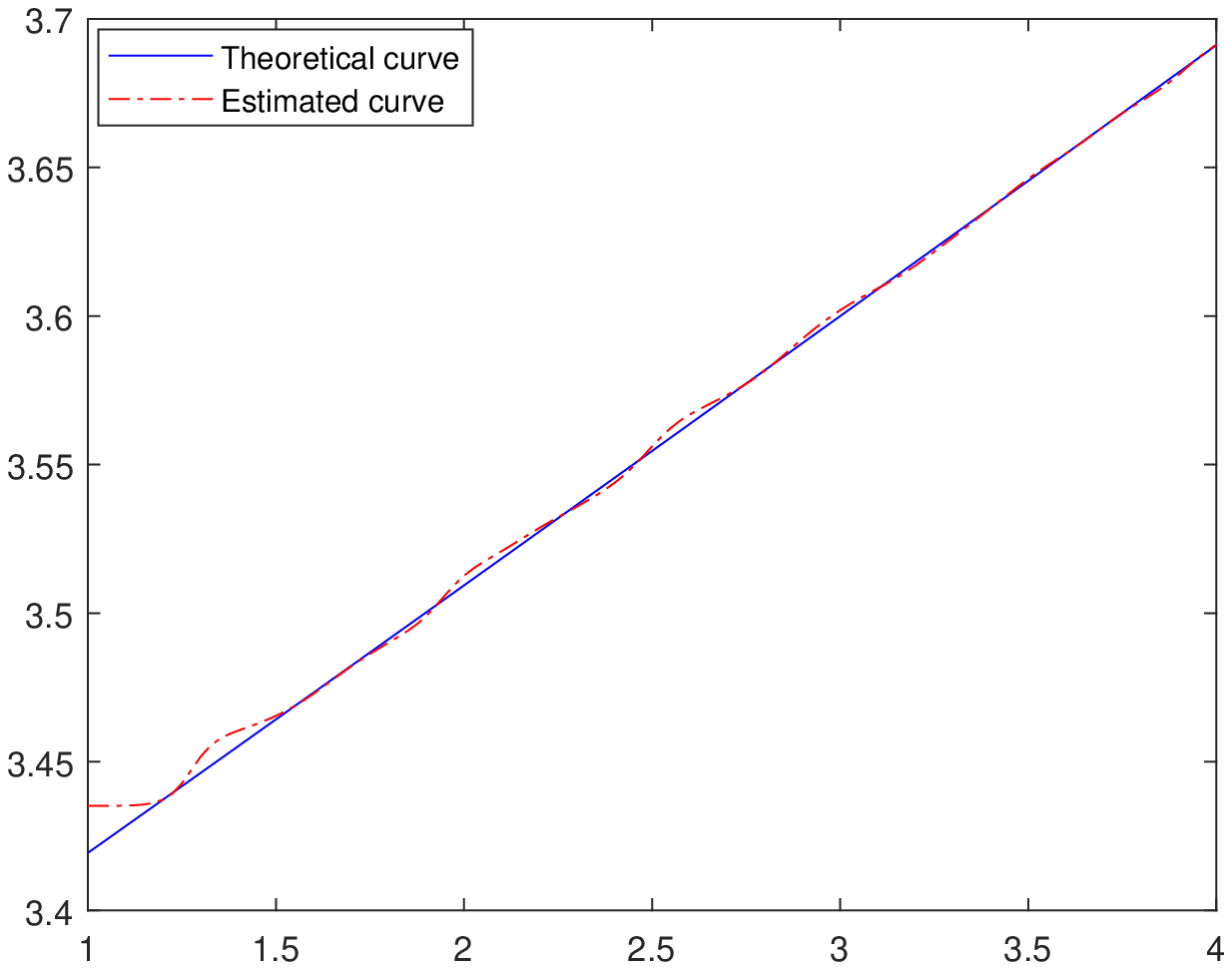}
	\end{minipage} \hfill
	\begin{minipage}[c]{.26\linewidth}
		\includegraphics[width=1.3\textwidth]{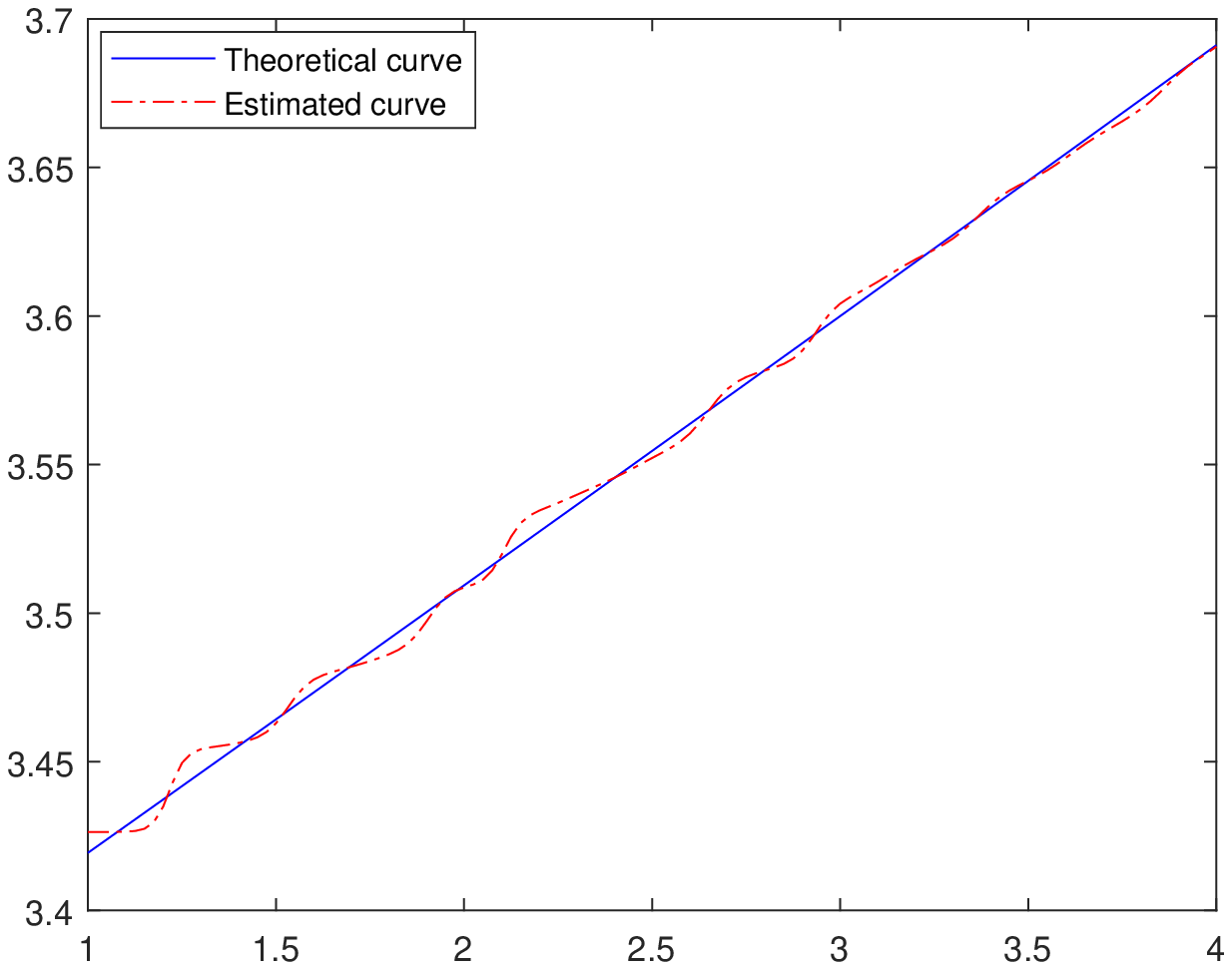}
	\end{minipage}\hfill\hfill
	\caption{ \textcolor{blue}{$m(x)$}, \textcolor{red}{$\widehat{m}(x)$} with $c=3$, $\rho=0.1$ and $n=300$ for C.P. $\approx 10, 33\;\text{and}\;72 \%\;$ respectively.}\label{figure2}
\end{figure}
\subsubsection{Strong dependency}
\paragraph{ $\bullet$ Effect of sample size:} For the case of highly dependent data $(\rho=0.7)$ and for a fixed C.P., we can observe from \hyperref[figure3]{Figure 3} that the RER estimator is adjusted to the theoretical curve when $n$ rises.
\begin{figure}[!h]
	\begin{minipage}[c]{.26\linewidth}
		\includegraphics[width=1.3\textwidth]{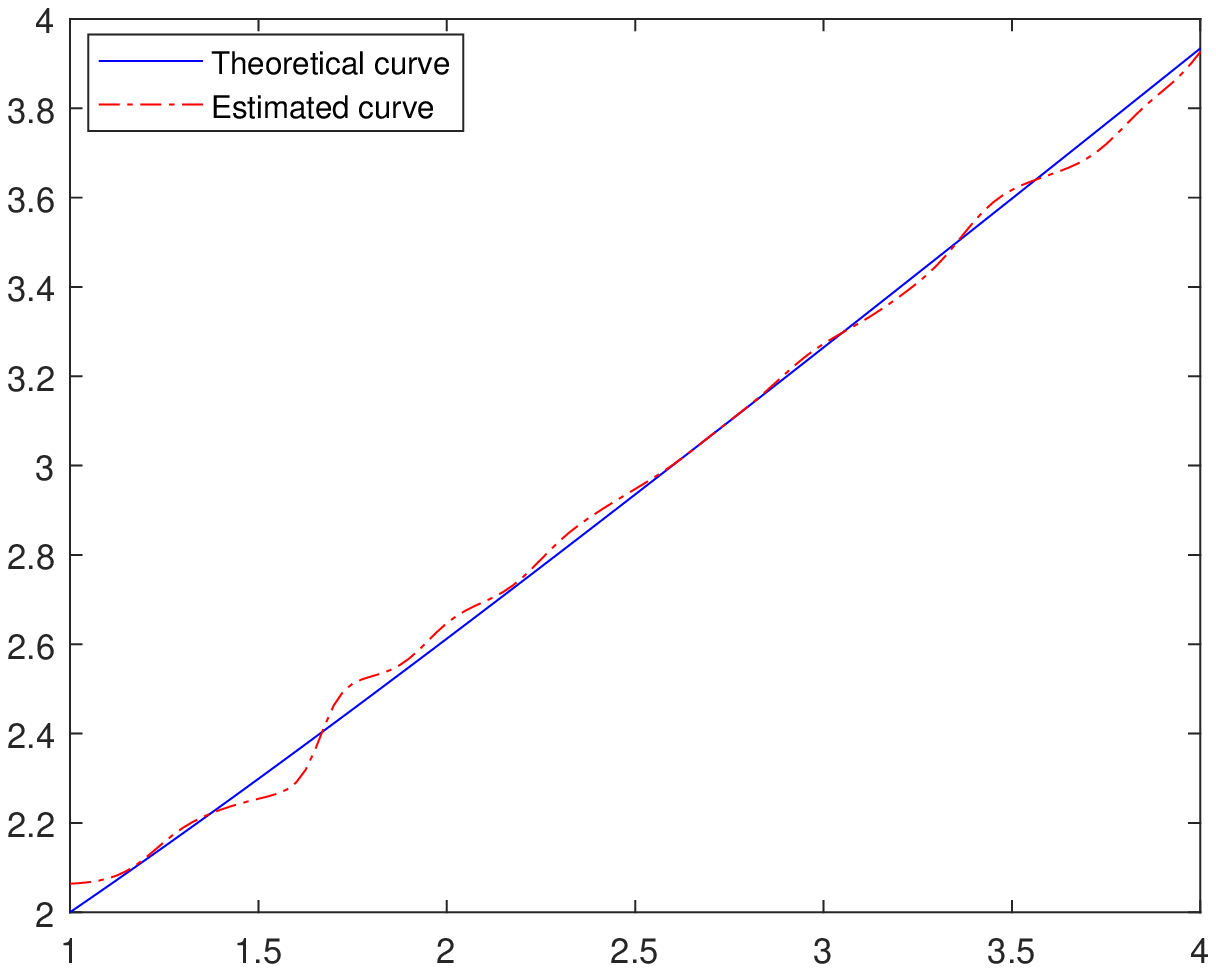}   
	\end{minipage} \hfill
	\begin{minipage}[c]{.26\linewidth}
		\includegraphics[width=1.3\textwidth]{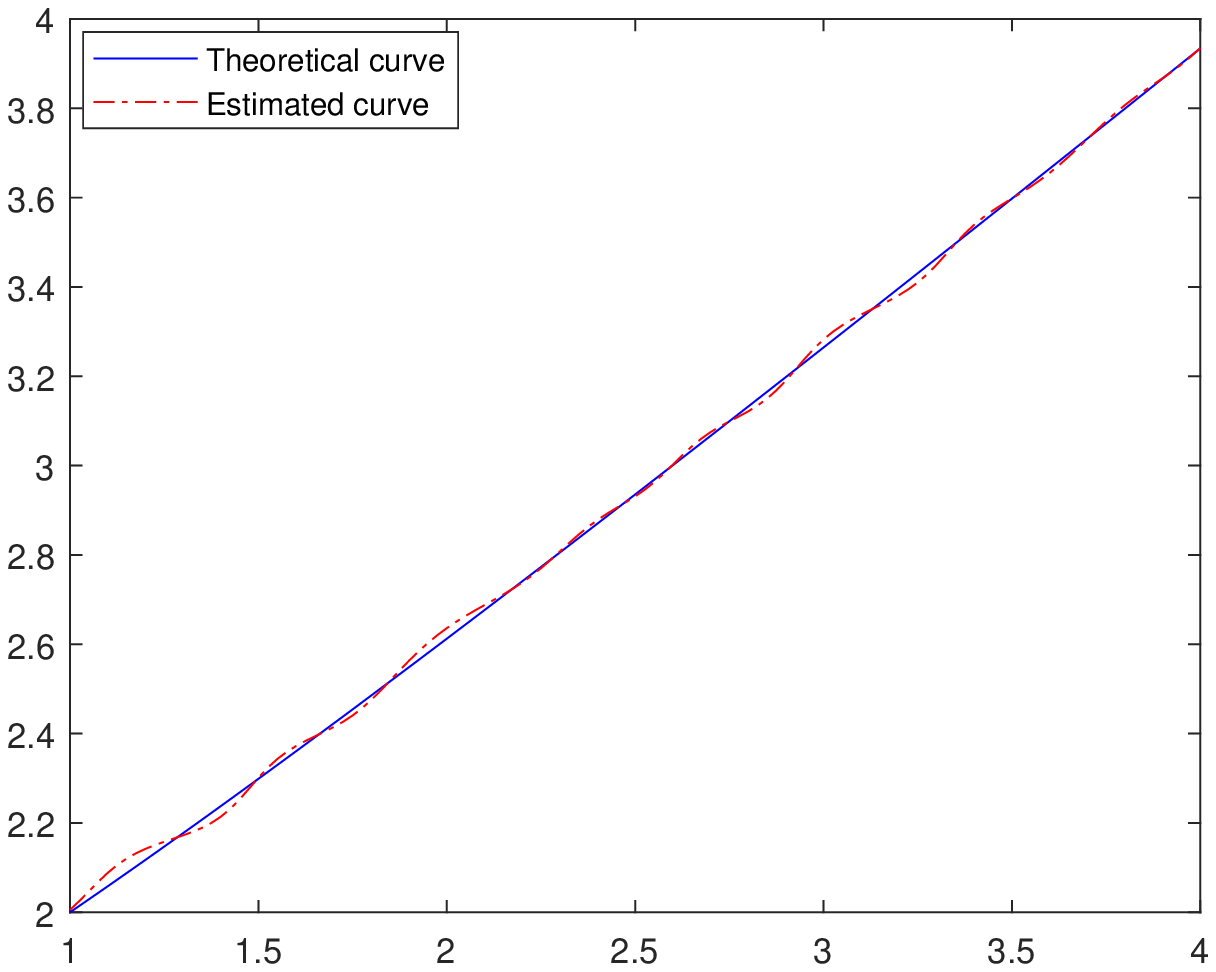}
	\end{minipage} \hfill
	\begin{minipage}[c]{.26\linewidth}
		\includegraphics[width=1.3\textwidth]{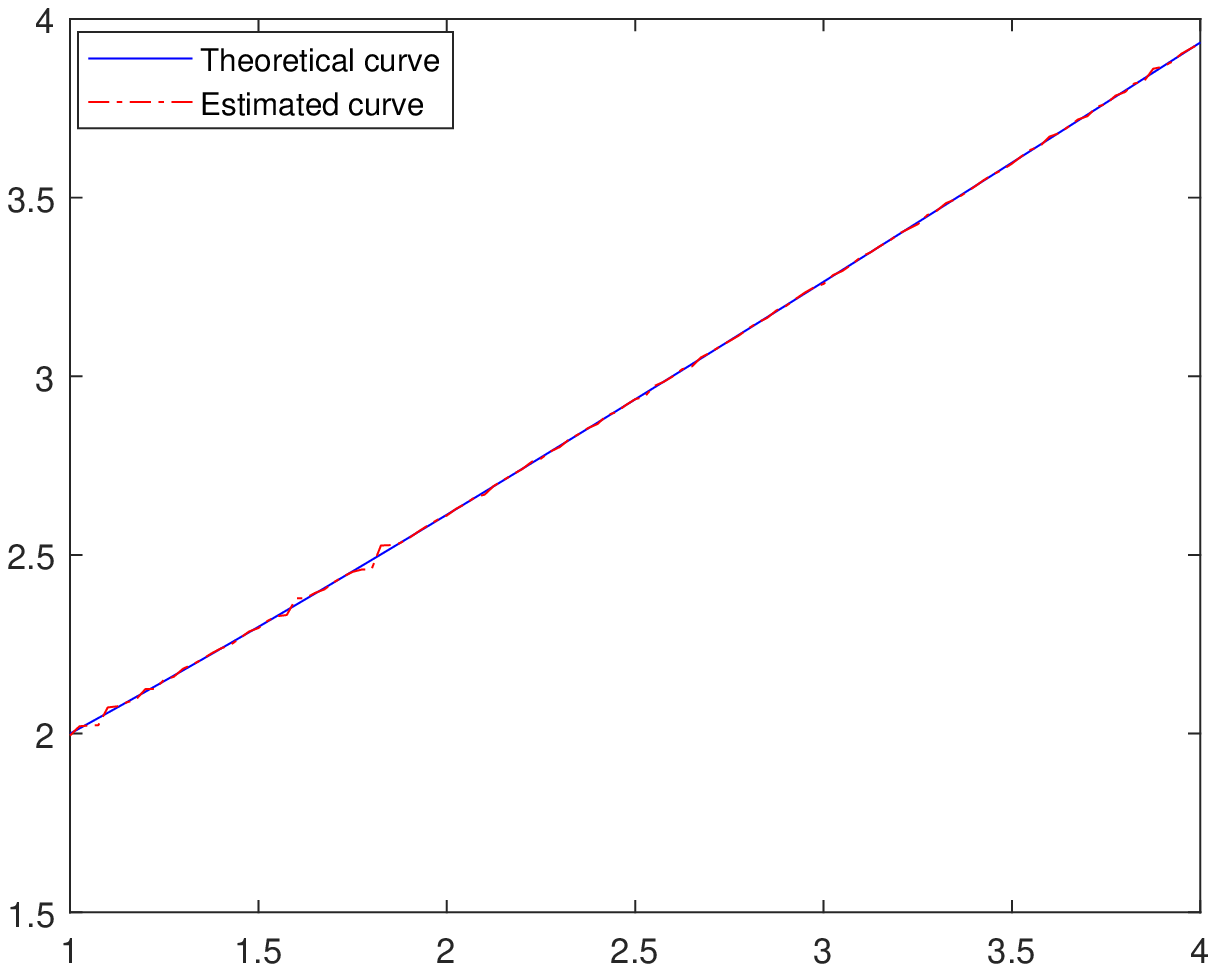}
	\end{minipage}\hfill\hfill
	\caption{\textcolor{blue}{$m(x)$}, \textcolor{red}{$\widehat{m}(x)$} with $c=1$, $\rho=0.7$ and C.P.$\approx 30\%$ for $n=100,300 \; \text{and} \;500$  respectively.}\label{figure3}
\end{figure}
\paragraph{$\bullet$ Effect of C.P.:} We see clearly that the quality of fit is better for large sample size and low percentage of censoring (see \hyperref[figure4]{Figure 4}). 
\begin{figure}[!h]
	\begin{minipage}[c]{.26\linewidth}
		\includegraphics[width=1.3\textwidth]{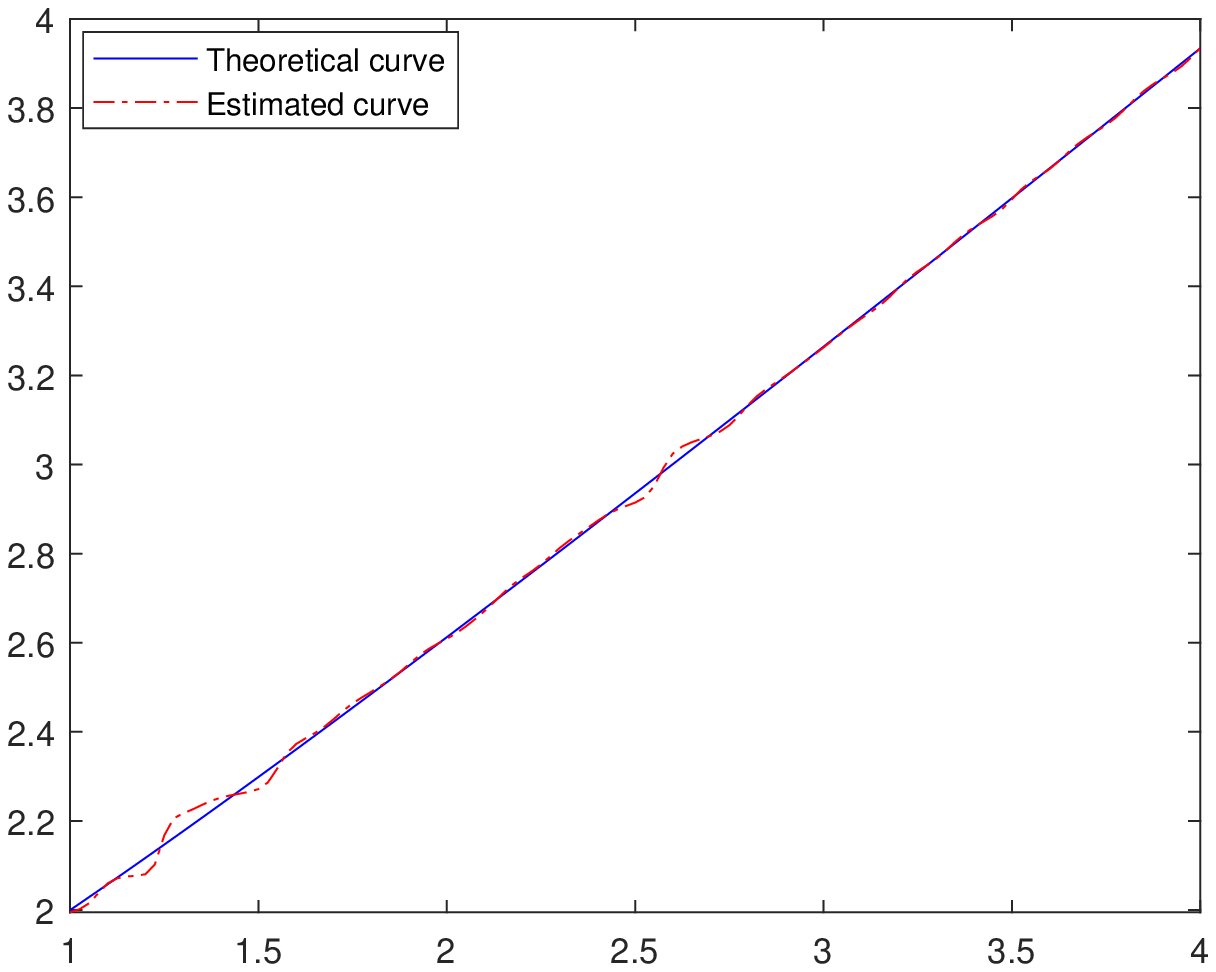}   
	\end{minipage} \hfill
	\begin{minipage}[c]{.26\linewidth}
		\includegraphics[width=1.3\textwidth]{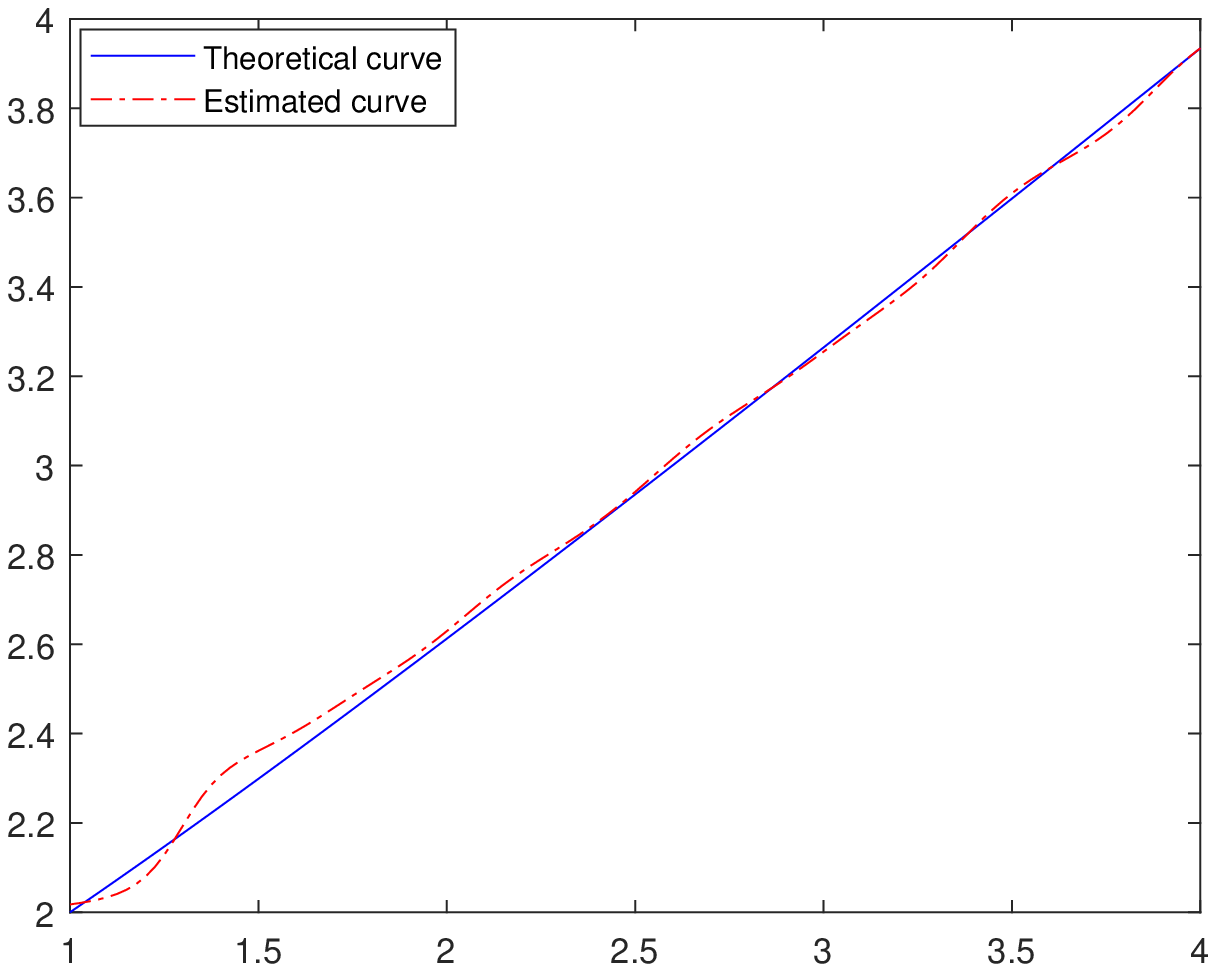}
	\end{minipage} \hfill
	\begin{minipage}[c]{.26\linewidth}
		\includegraphics[width=1.3\textwidth]{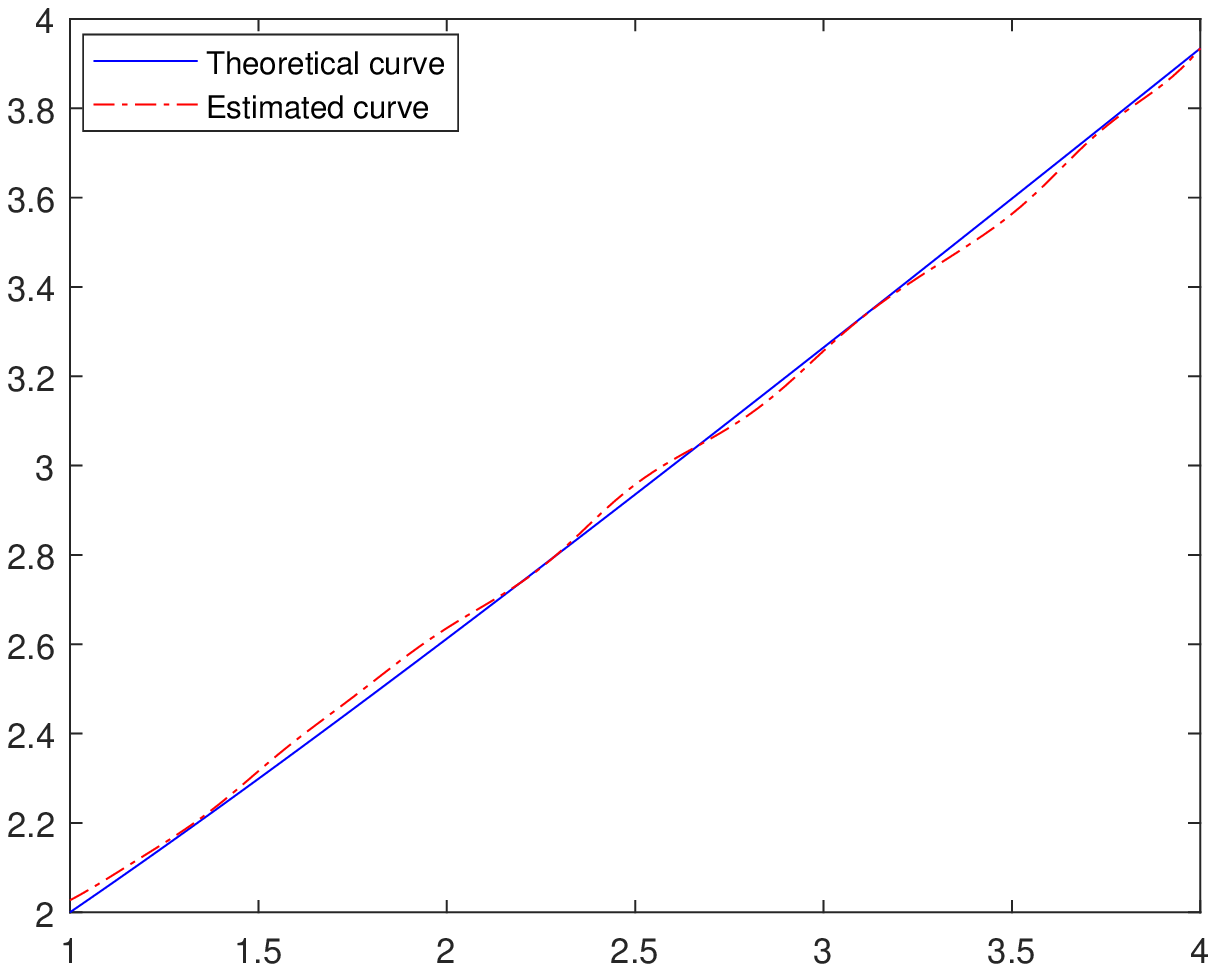}
	\end{minipage}\hfill\hfill
	\caption{\textcolor{blue}{$m(x)$}, \textcolor{red}{$\widehat{m}(x)$} with $c=1$, $\rho=0.7$ and $n=300$ for C.P.$\approx 11, 33\; \text{and} \;65\%$ respectively.}\label{figure4}
\end{figure}
\subsection{Nonlinear case}
We consider now, three nonlinear functions: 
\begin{align*}
T_i&=1+cos\left(\frac{\pi}{2}X_i\right), \;\;\;\text{Cosinus model,}\\
T_i&=\exp{(\rho^2 X_i)}, \;\;\;\;\;\;\;\;\text{Exponential model},\\
T_i&=\frac{1}{X_i}, \;\;\;\;\;\;\;\;\;\;\;\;\;\;\;\;\;\;\;\;\text{Inverse model.}
\end{align*}
\hyperref[figure5]{Figure 5} shows that the quality of the fit is good as in linear model. Clearly, we see that the adjustment improves when $n$ increases.
\begin{figure}[!h]
	\begin{minipage}[c]{.26\linewidth}
		\includegraphics[width=1.3\textwidth]{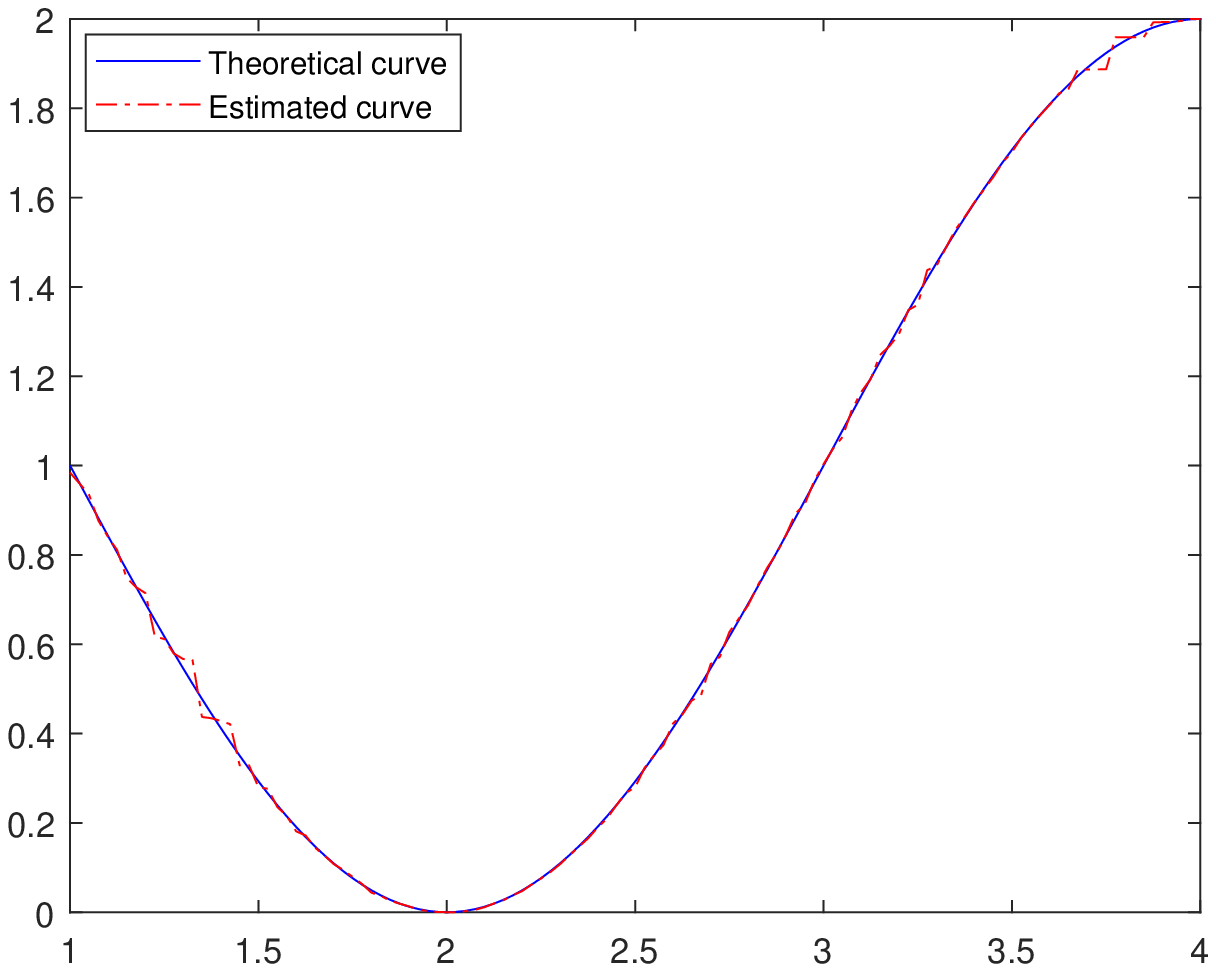}   
	\end{minipage} \hfill
	\begin{minipage}[c]{.26\linewidth}
		\includegraphics[width=1.3\textwidth]{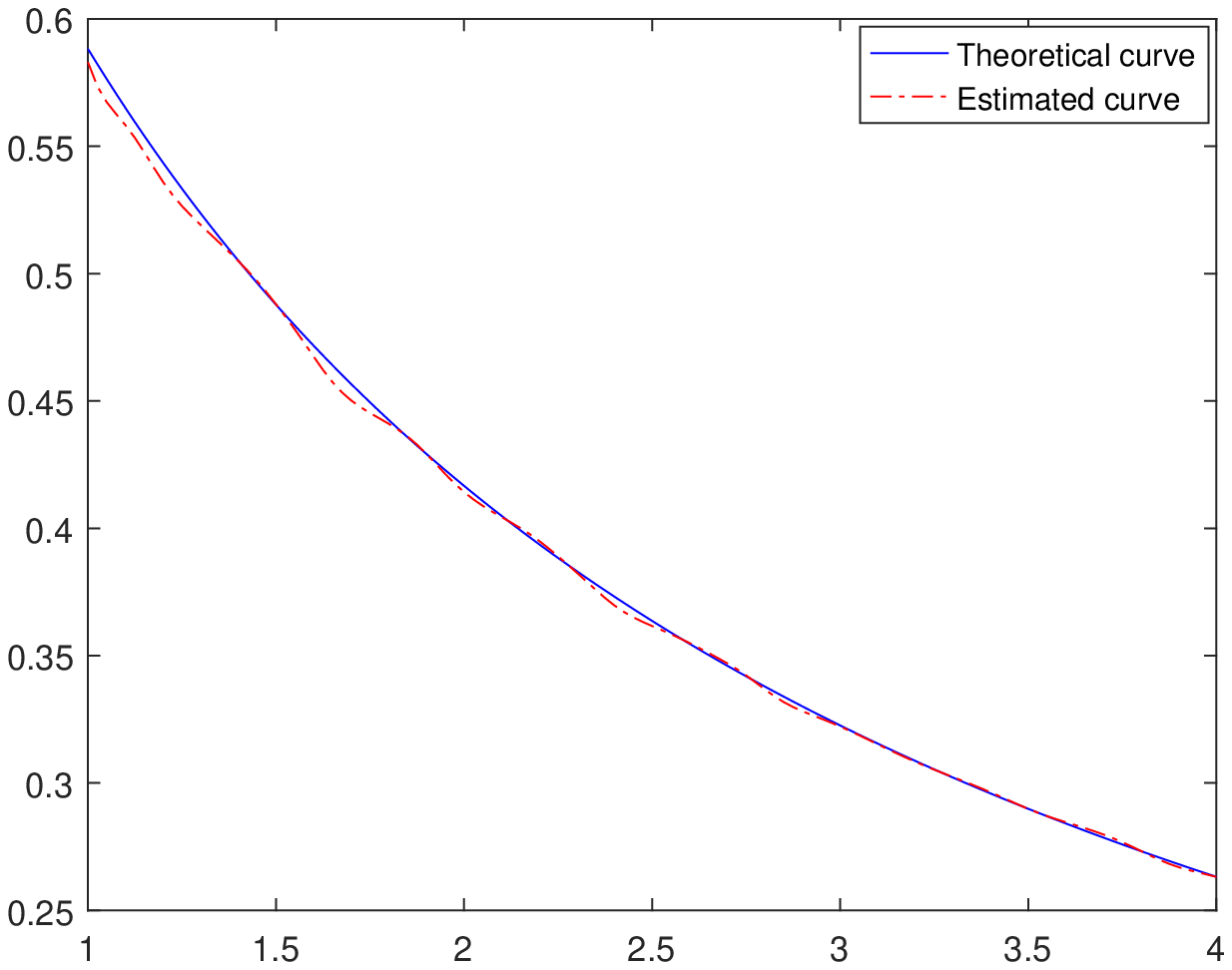}
	\end{minipage} \hfill
	\begin{minipage}[c]{.26\linewidth}
		\includegraphics[width=1.3\textwidth]{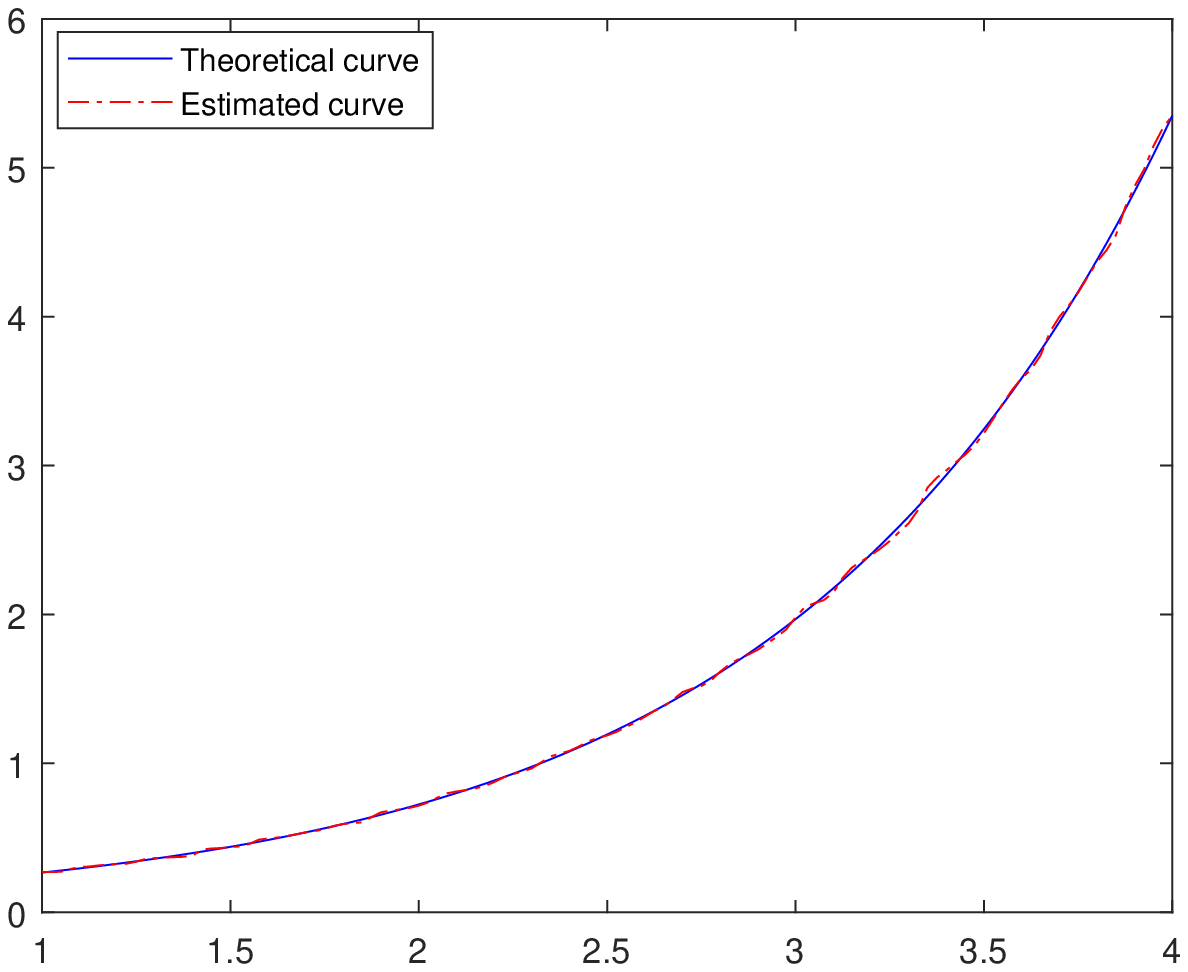}
	\end{minipage}\hfill\hfill
	\caption{\textcolor{blue}{$m(x)$}, \textcolor{red}{$\widehat{m}(x)$} with $c=1$, $\rho=0.7$, C.P.$ \approx 25\%$ and $n=300$}\label{figure5}
\end{figure}
\subsection{Effect of outliers}
To show the robustness of our approach, we generate the case where the data contains outliers. To create this outlier effect, $20$ values of this sample are multiplied by a factor called M.F.. From \hyperref[figure6]{Figure 6}, we can see that our estimator is close to the theoretical curve knowing that we observe only $70\%$ of the trues values. Then, it is absolutely clear that our estimator is resistant in the presence of outliers.
\begin{figure}[!h]
	\begin{minipage}[c]{.26\linewidth}
		\includegraphics[width=1.3\textwidth]{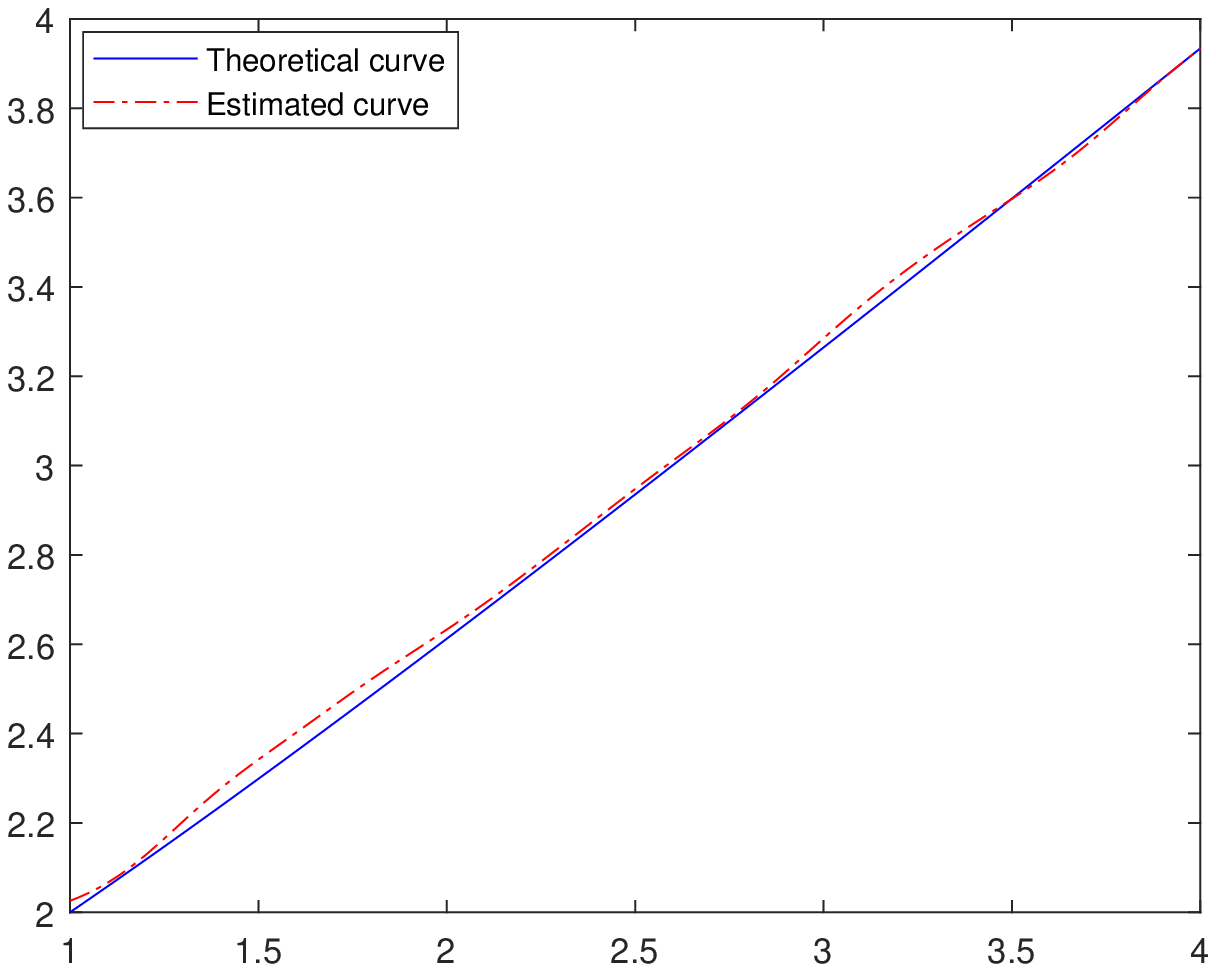}   
	\end{minipage} \hfill
	\begin{minipage}[c]{.26\linewidth}
		\includegraphics[width=1.3\textwidth]{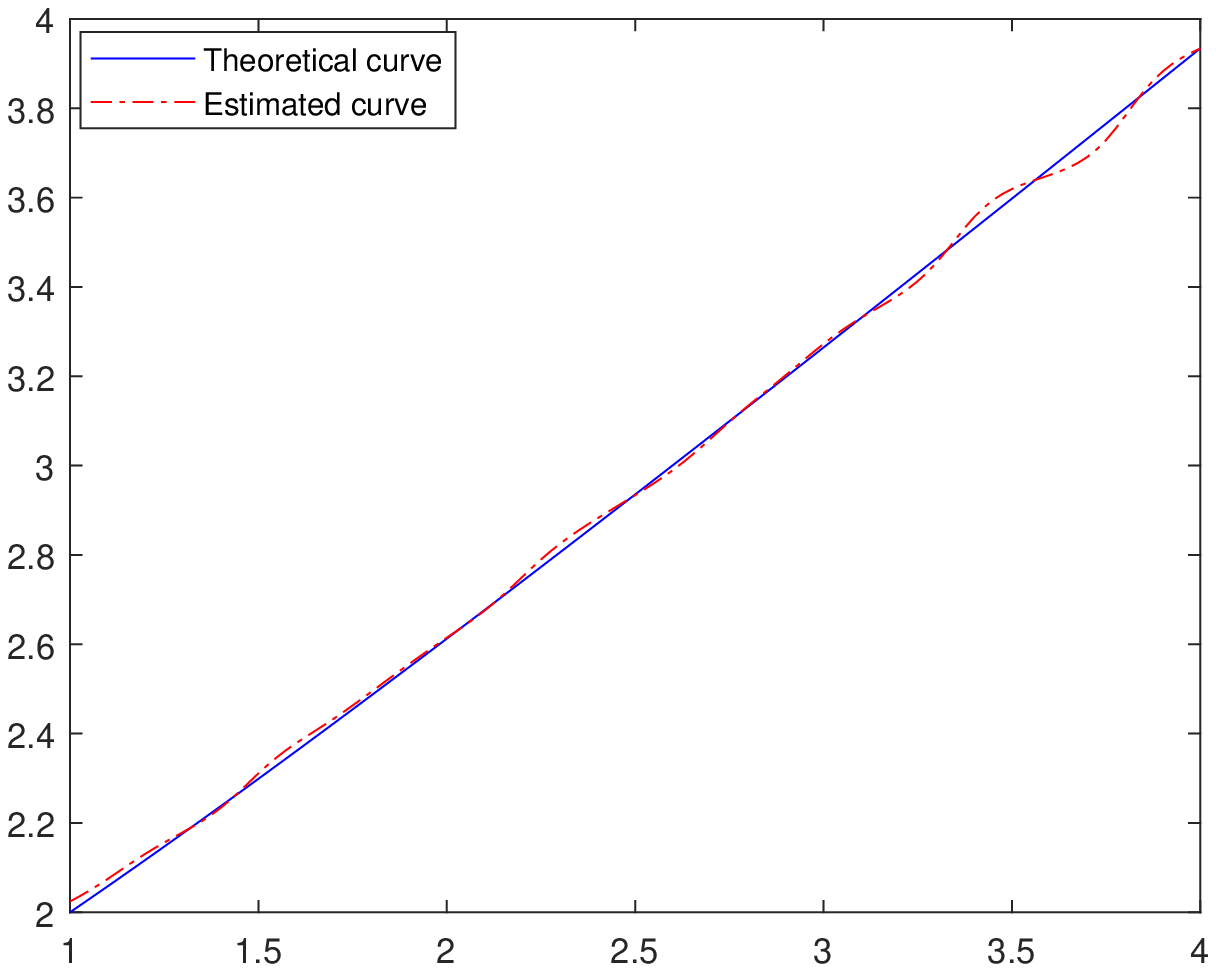}
	\end{minipage} \hfill
	\begin{minipage}[c]{.26\linewidth}
		\includegraphics[width=1.3\textwidth]{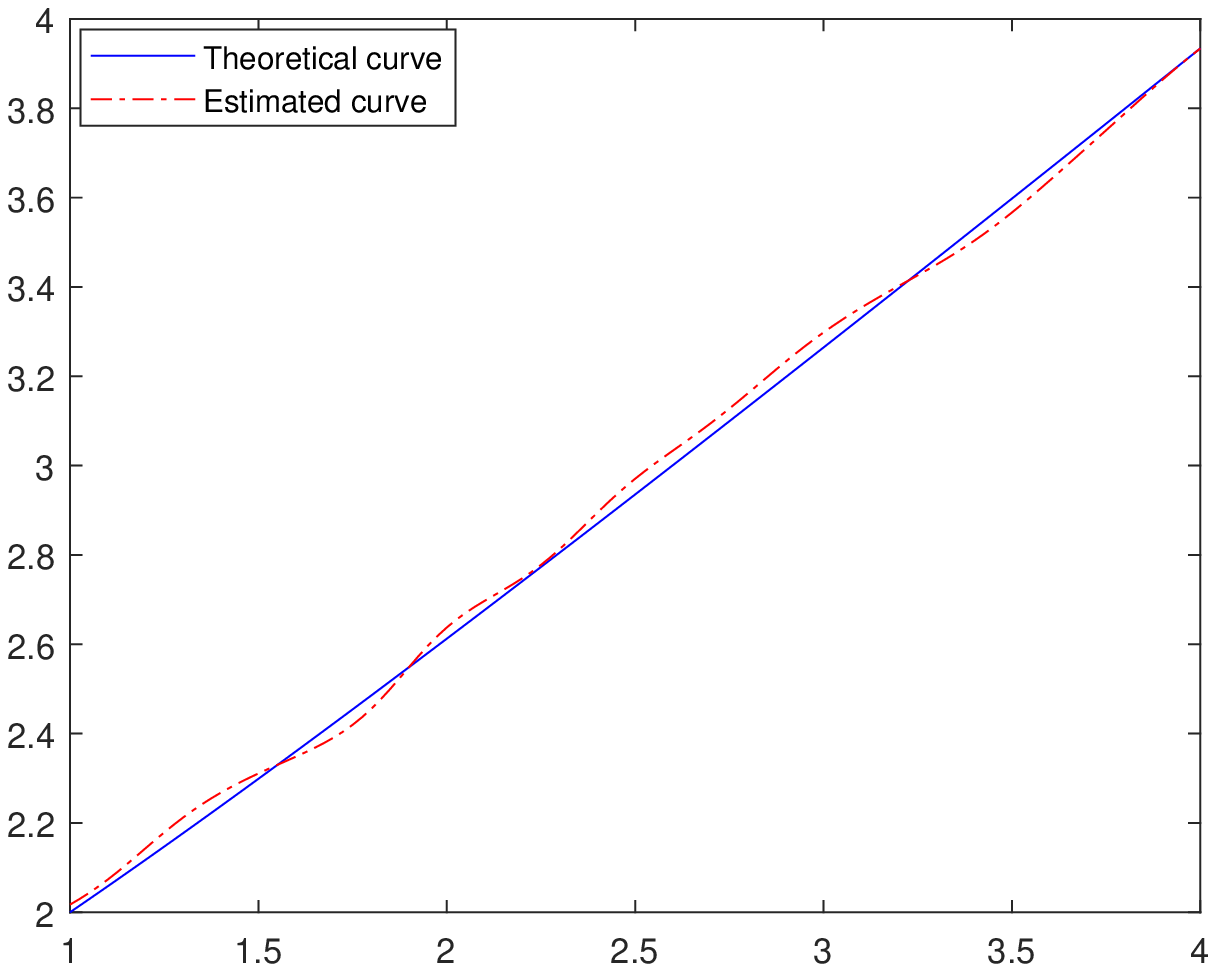}
	\end{minipage}\hfill\hfill
	\caption{\textcolor{blue}{$m(x)$}, \textcolor{red}{$\widehat{m}(x)$}, with $c=1$, $\rho=0.7$, C.P. $\approx 30\%$ $n=300$ and M.F.$=10, 50\; \text{and} \; 100$ respectively.}\label{figure6}
\end{figure}
\subsection{Effect of contamination of the random error $\varepsilon$:} We take the same algorithm as before by changing step 1 which becomes : 
\begin{itemize}
	\item[$\bullet$] Step 1'. $\varepsilon_i \leadsto (1-\beta) \eta_1+\alpha\eta_2$ where $\eta_1 \leadsto \mathcal{N}(0,1)$ and $\eta_2 \leadsto \mathcal{N}(0,\lambda)$. We choose the level of contamination $\beta=0.01,0.05 \;\text{and}\;0.1$ and the magnitude of contamination $\lambda=3$ generally.
\end{itemize}
We observe from (\hyperref[figure7]{Figure 7}) that the quality of the adjustment to the theoretical function deteriorates when the level of contamination $\alpha$ becomes higher.
\begin{figure}[!h]
	\begin{minipage}[c]{.26\linewidth}
		\includegraphics[width=1.3\textwidth]{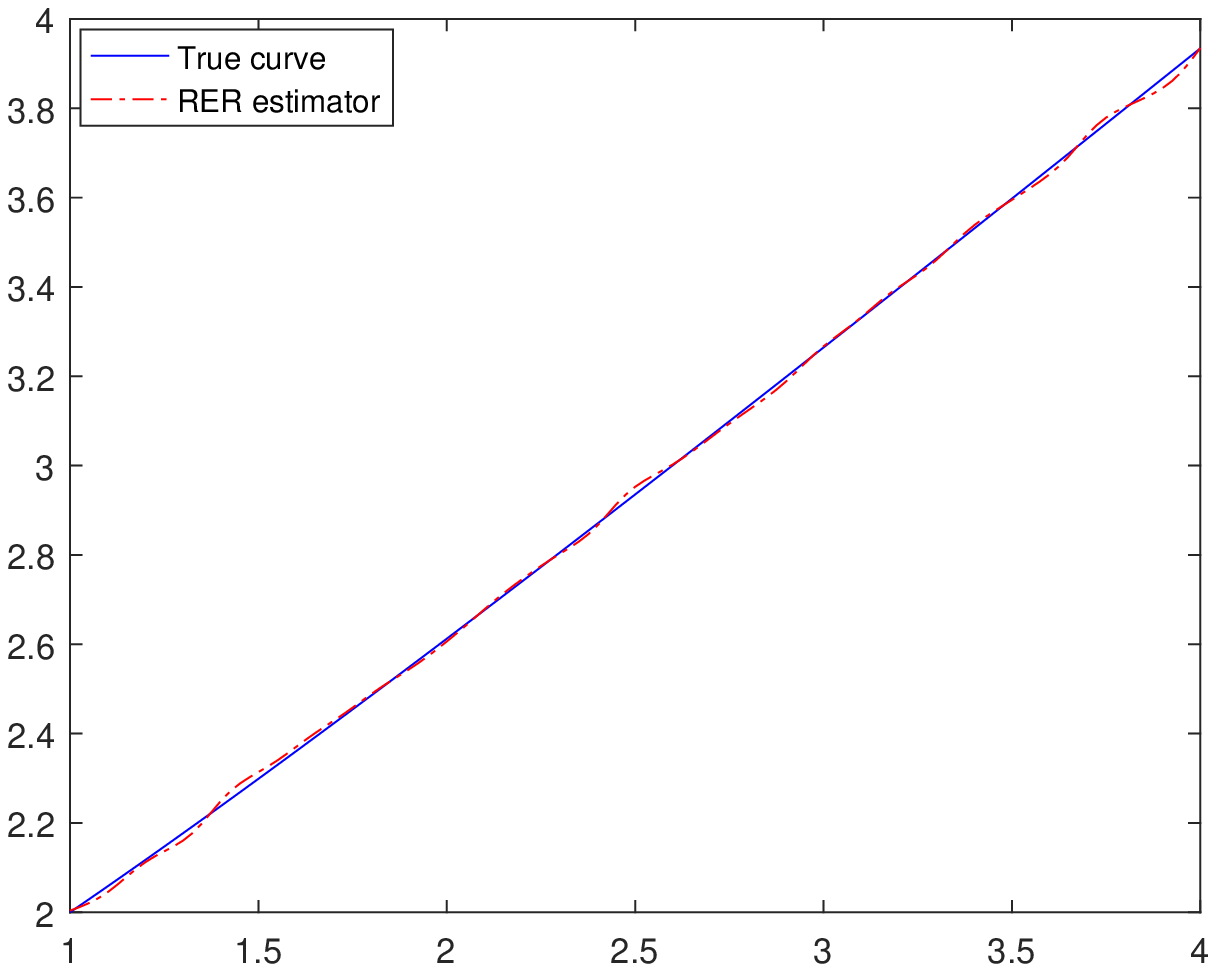}   
	\end{minipage} \hfill
	\begin{minipage}[c]{.26\linewidth}
		\includegraphics[width=1.3\textwidth]{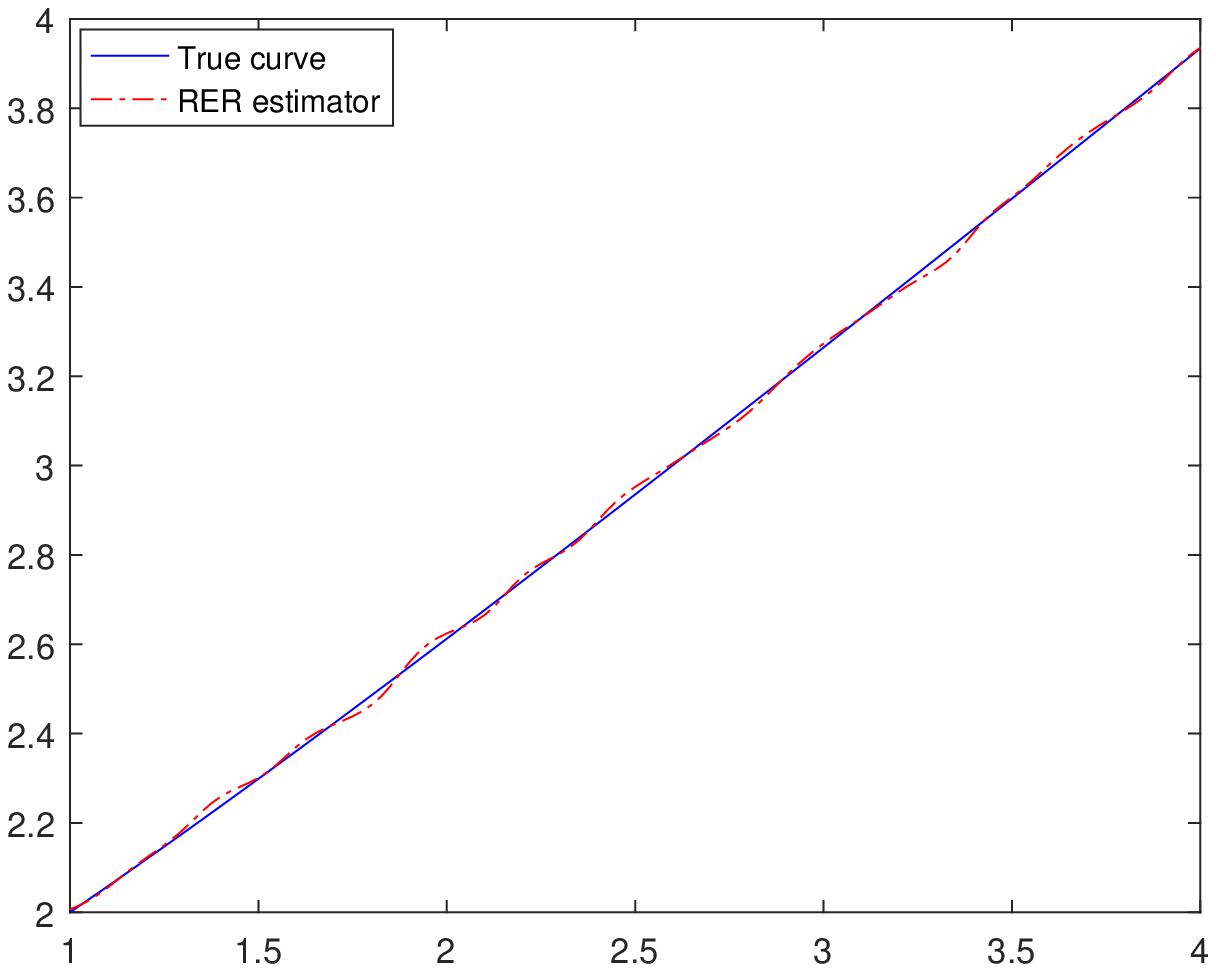}
	\end{minipage} \hfill
	\begin{minipage}[c]{.26\linewidth}
		\includegraphics[width=1.3\textwidth]{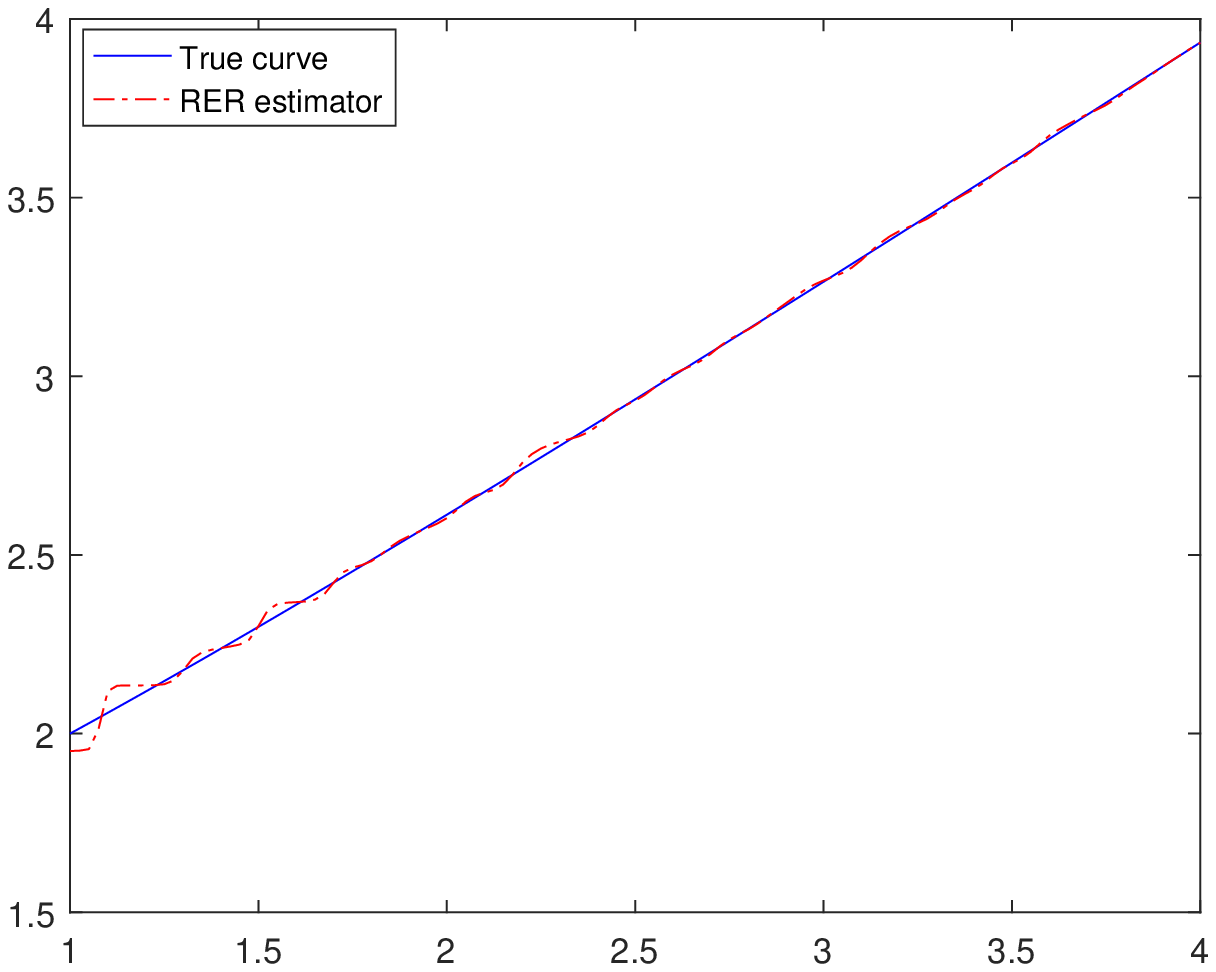}
	\end{minipage}\hfill\hfill
	\caption{\textcolor{blue}{$m(x)$}, \textcolor{red}{$\widehat{m}(x)$}, $c=1$, $\rho=0.7$, $n=300$, C.P. $50\%$ and $\beta=0.01,0.05 \;\text{and}\; 0.1$ respectively.}\label{figure11}
\end{figure}
\subsection{Comparaison study}
To show the efficiency of the RER estimator, we carry out a comparative study in which we consider the classical regression (CR) estimator defined in \hyperref[guessoum]{Guessoum and Ould Sa\"id (2010)} by
\begin{equation}\label{CR}
\mu_n(x)=\frac{\displaystyle{\sum_{i=1}^{n}\frac{\delta_i Y_i}{\bar{G}_n(Y_i)}K_d(x-X_i)}}{\displaystyle{\sum_{i=1}^{n}K_d(x-X_i)}},
\end{equation}
for weak and strong dependency.
\subsubsection{Weak dependency}
\paragraph{$\bullet$ Effect of C.P.:} We fix $n$ and we vary the censoring rate. We can notice clearly (from \hyperref[figure8]{Figure 8}) that the RER estimator is near to the theoretical curve whereas the CR curve is distant from the true curve when C.P. increases. 
\begin{figure}[!h]
	\begin{minipage}[c]{.26\linewidth}
		\includegraphics[width=1.3\textwidth]{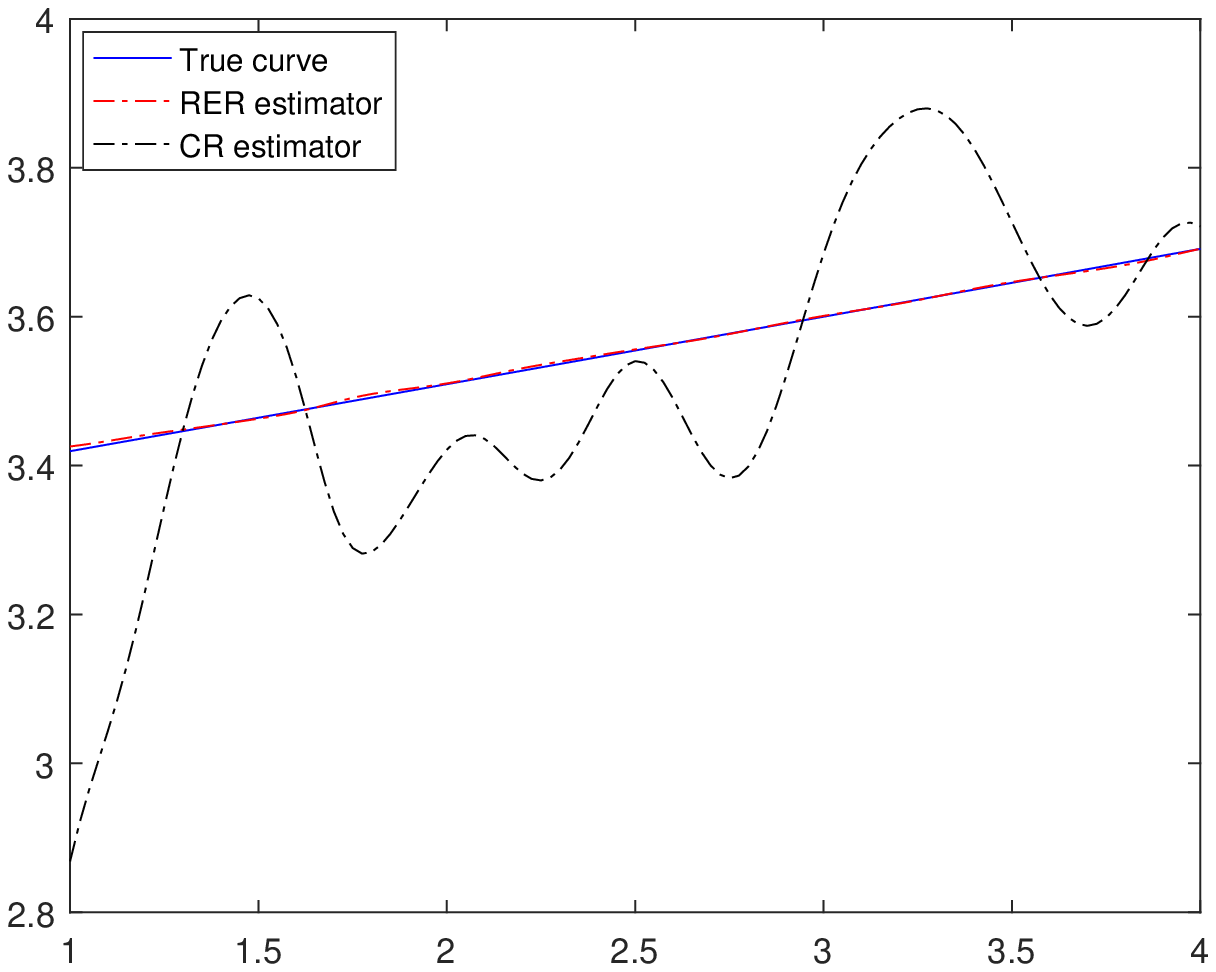}   
	\end{minipage} \hfill
	\begin{minipage}[c]{.26\linewidth}
		\includegraphics[width=1.3\textwidth]{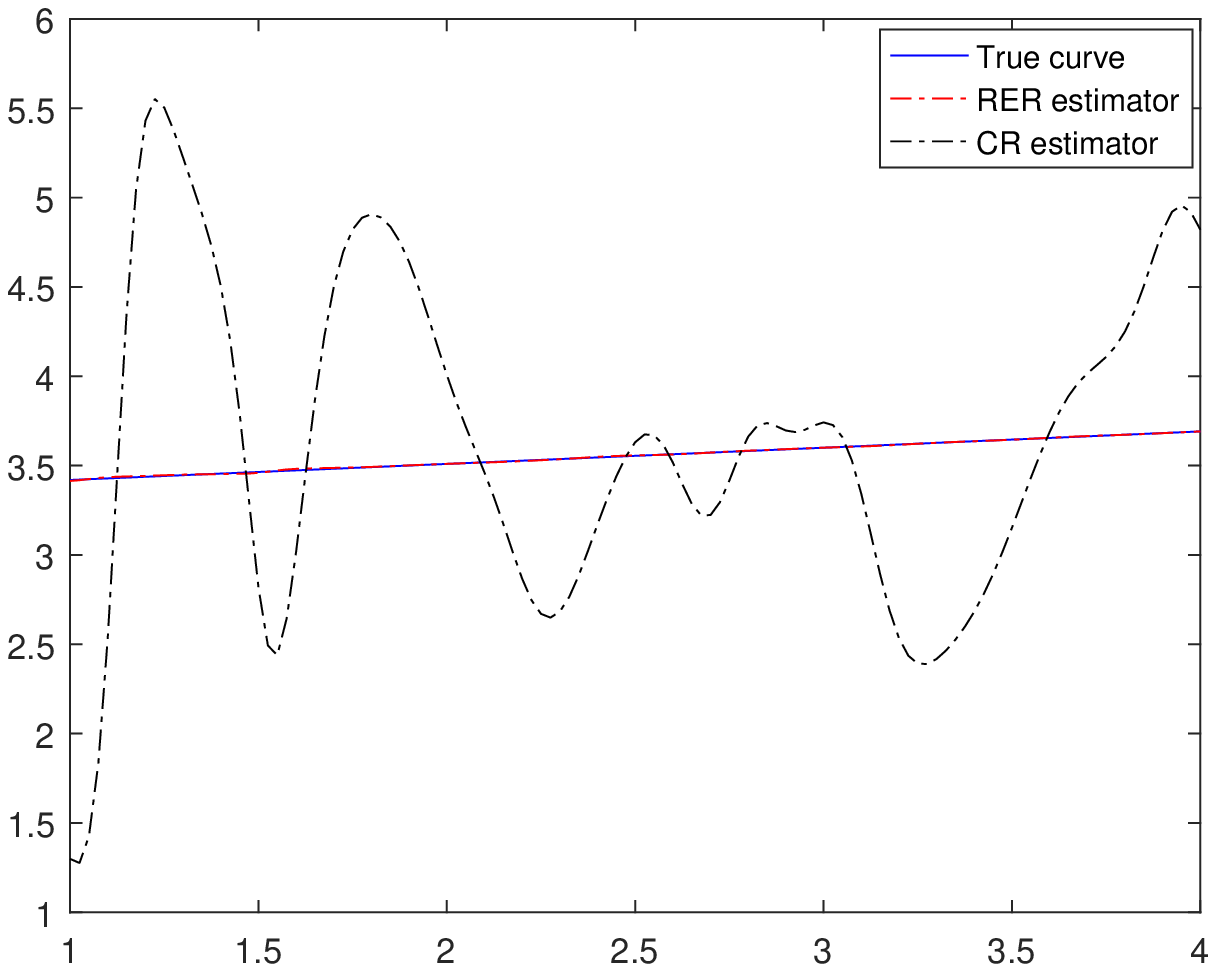}
	\end{minipage} \hfill
	\begin{minipage}[c]{.26\linewidth}
		\includegraphics[width=1.3\textwidth]{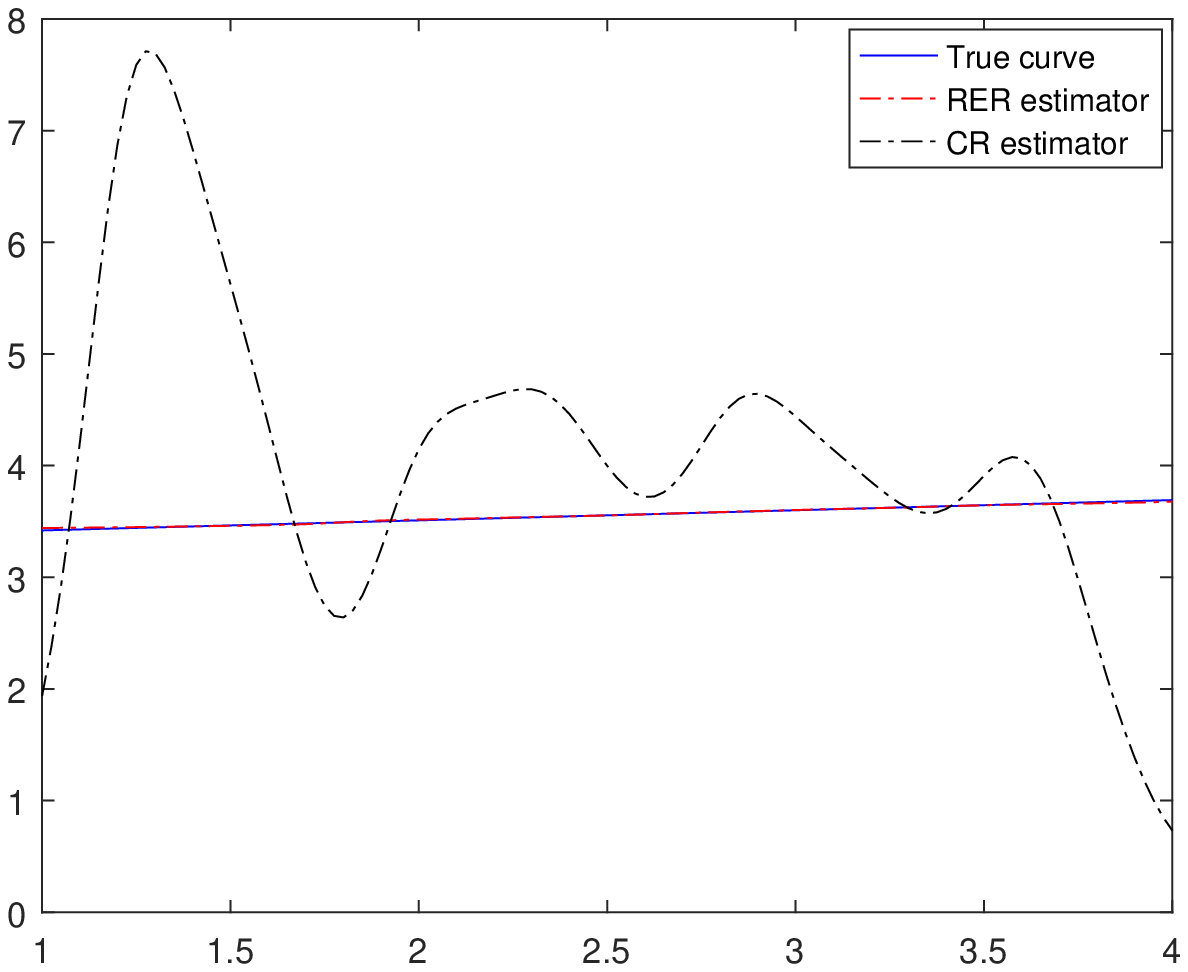}
	\end{minipage}\hfill\hfill
	\caption{\textcolor{blue}{$m(x)$}, \textcolor{red}{$\widehat{m}(x)$}, $\mu_n(x)$ with $c=3$, $\rho=0.1$, $n=300$ and C.P. $\approx 10, 50 \;\text{and}\; 80\%$ respectively.}\label{figure7}
\end{figure}
\paragraph{$\bullet$ Effect of outliers:} We fix $n$, C.P. and we vary the M.F. It can be reported from \hyperref[figure9]{Figure 9} that the RER estimator is overlapped on the true curve in contrast with the CR estimator which is significantly affected by the M.F. when the dependency is weak.
\begin{figure}[!h]
	\begin{minipage}[c]{.26\linewidth}
		\includegraphics[width=1.3\textwidth]{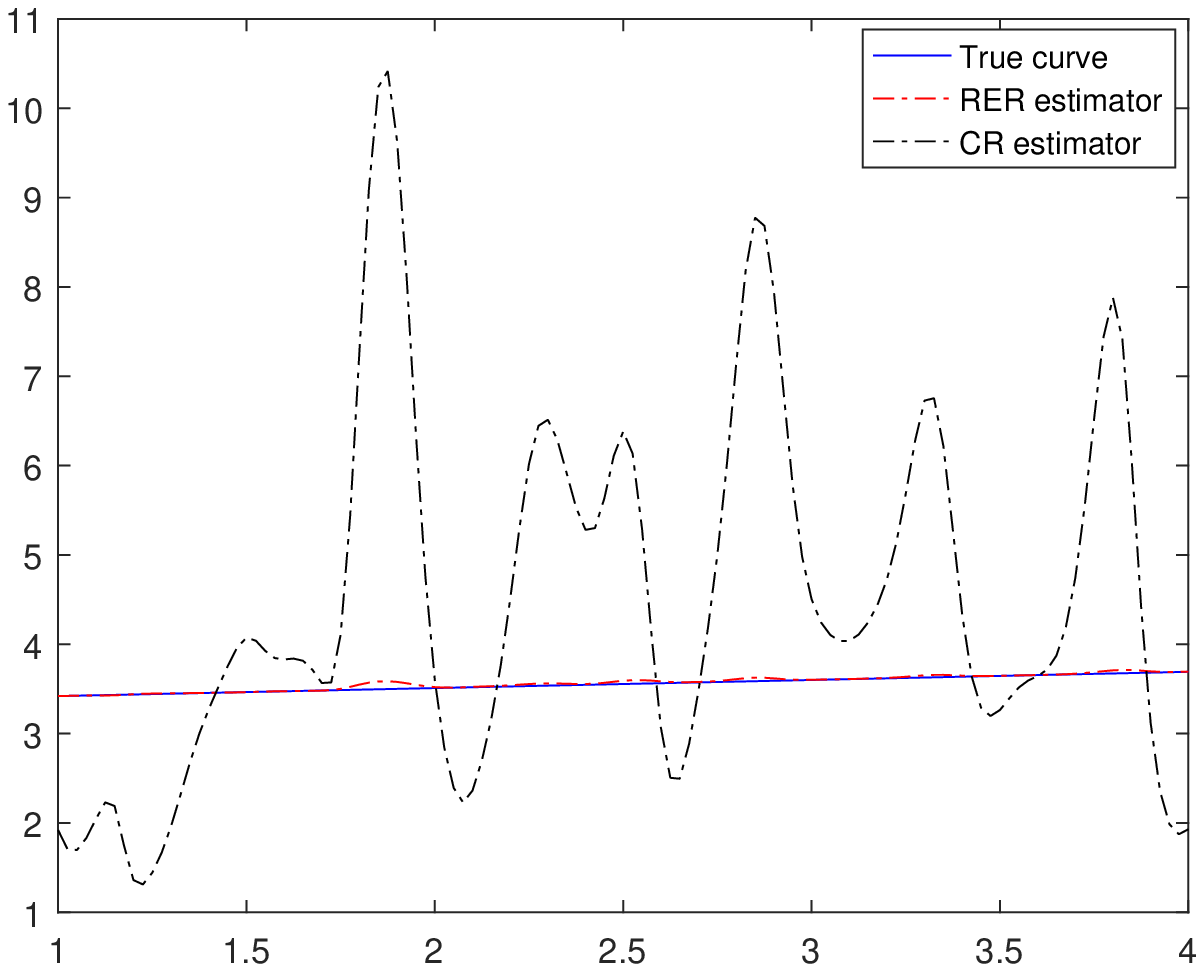}
	\end{minipage} \hfill
	\begin{minipage}[c]{.26\linewidth}
		\includegraphics[width=1.3\textwidth]{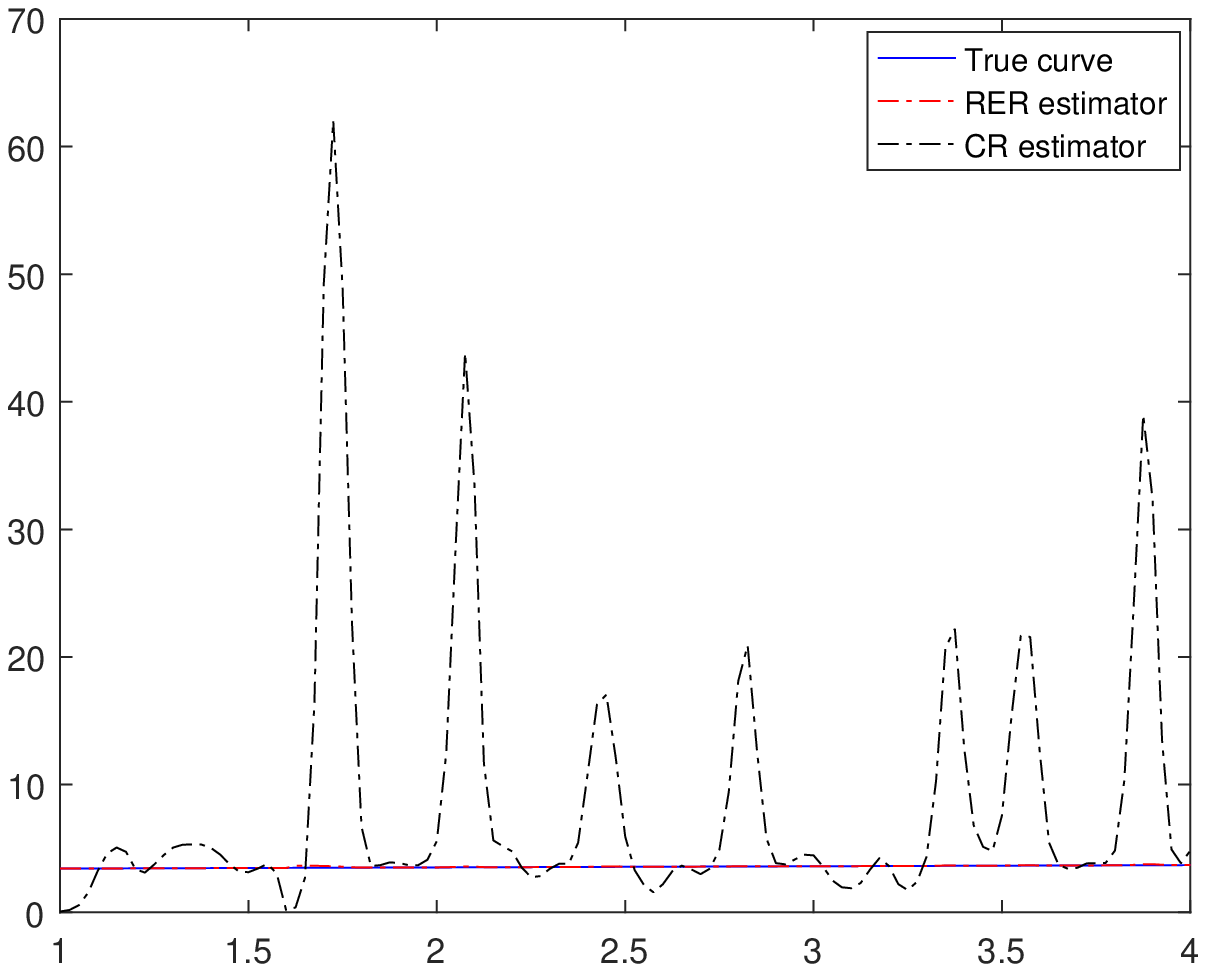}
	\end{minipage} \hfill
	\begin{minipage}[c]{.26\linewidth}
		\includegraphics[width=1.3\textwidth]{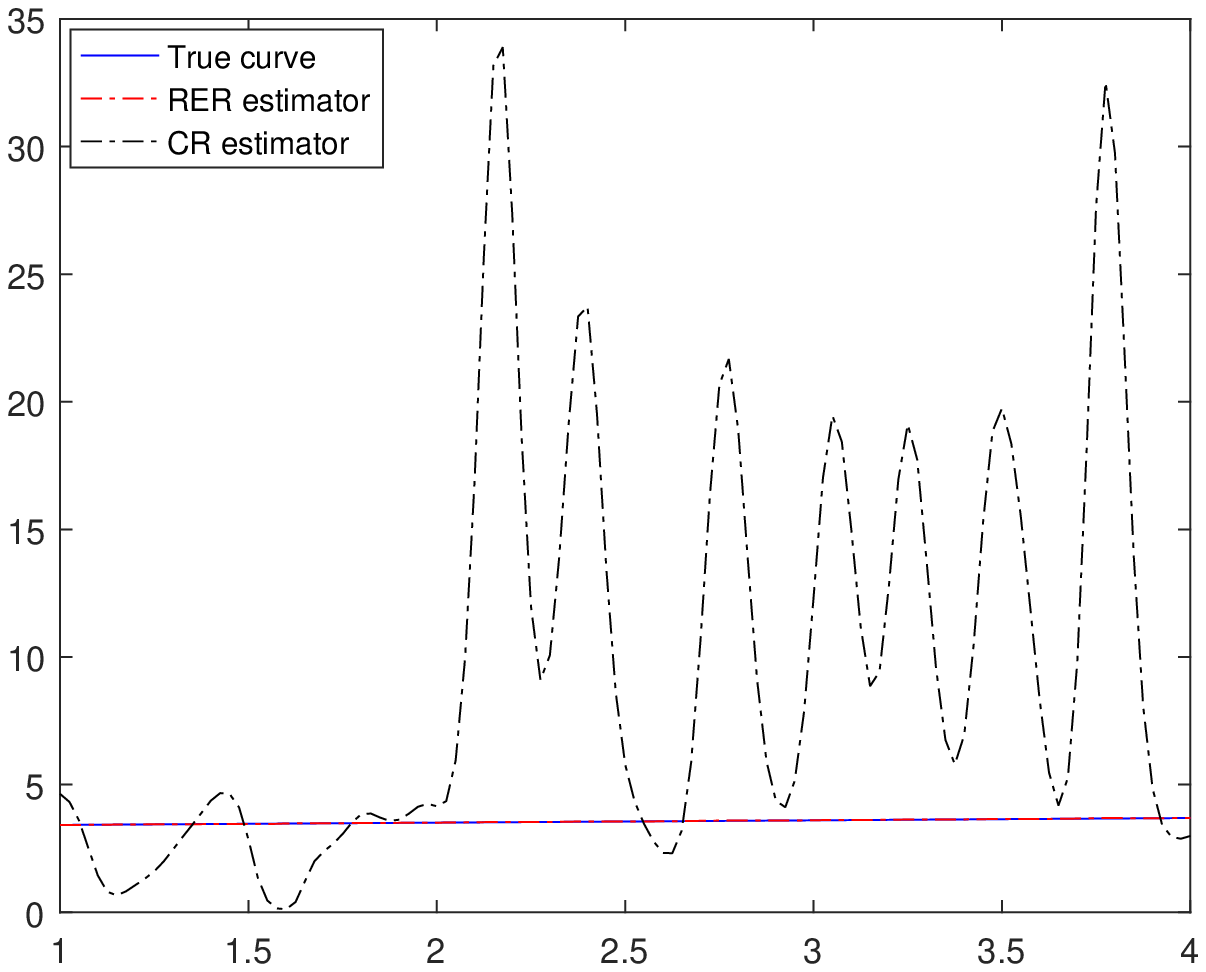}
	\end{minipage}\hfill\hfill
	\caption{\textcolor{blue}{$m(x)$}, \textcolor{red}{$\widehat{m}(x)$}, $\mu_n(x)$ with $c=3$, $\rho=0.1$, $n=300$, C.P. $\approx 35$ and M.F.=$10, 25 \;\text{and}\; 50\%$ respectively.}\label{figure8}
\end{figure}
\subsubsection{Strong dependency}
\paragraph{$\bullet$ Effect of C.P.:} We fix $\rho$, $n$ and we vary the C.P. to examine the effect of censorship on both RER and CR estimators when the dependency is strong. We can observe from \hyperref[figure10]{Figure 10} that the RER estimator 
\begin{figure}[!h]
	\begin{minipage}[c]{.26\linewidth}
		\includegraphics[width=1.3\textwidth]{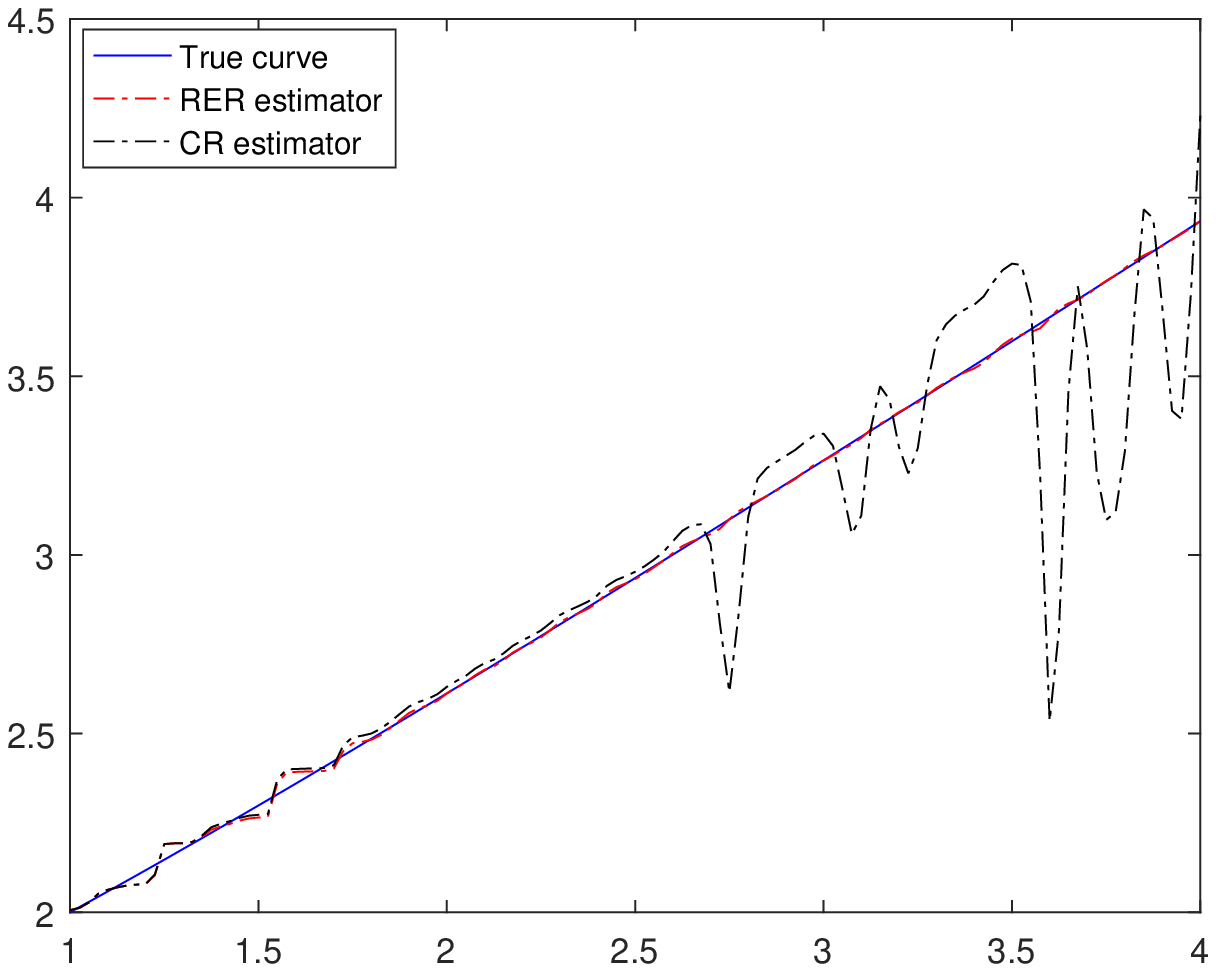}   
	\end{minipage} \hfill
	\begin{minipage}[c]{.26\linewidth}
		\includegraphics[width=1.3\textwidth]{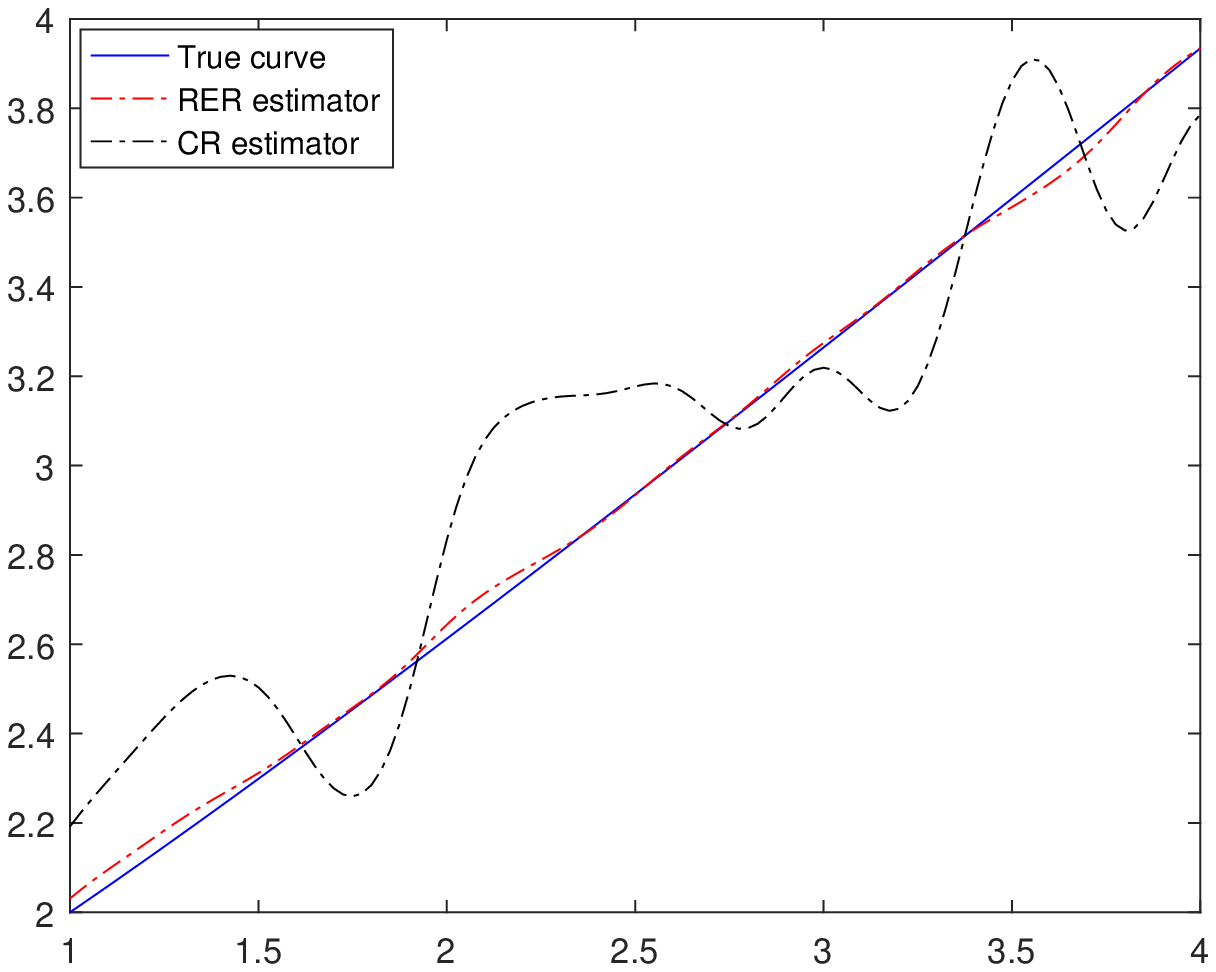}
	\end{minipage} \hfill
	\begin{minipage}[c]{.26\linewidth}
		\includegraphics[width=1.3\textwidth]{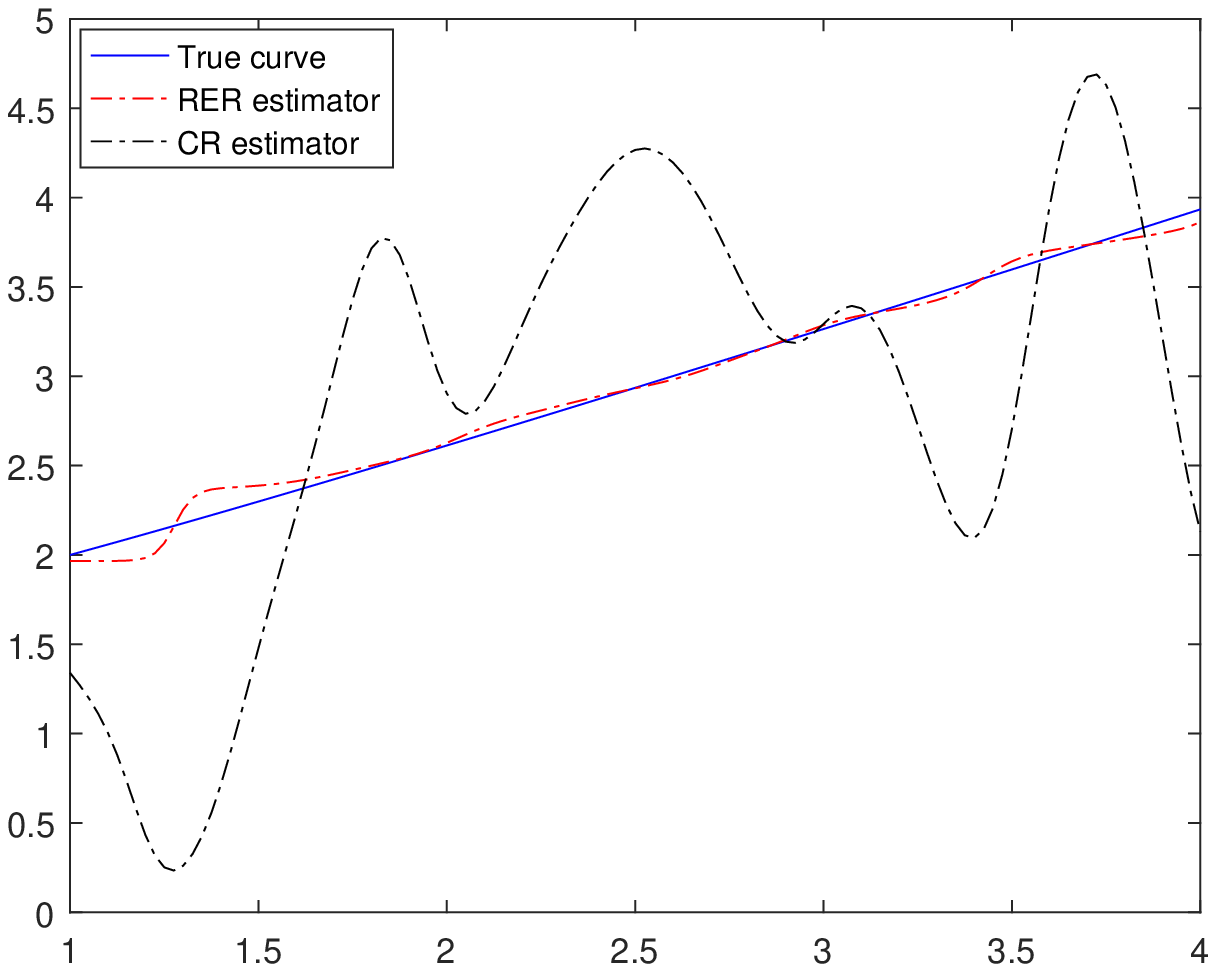}
	\end{minipage}\hfill\hfill
	\caption{\textcolor{blue}{$m(x)$}, \textcolor{red}{$\widehat{m}(x)$}, $\mu_n(x)$ with $c=1$, $\rho=0.7$, $n=300$ and C.P. $\approx 10, 50 \;\text{and}\; 80\%$ respectively.}\label{figure9}
\end{figure}
\paragraph{$\bullet$ Effect of outliers:} We fix $\rho$, $n$, C.P. and we vary the M.F. to evaluate the effect of outliers on both estimators when the dependency is high. As expected, our estimator remains resistant to outliers under a high dependency unlike that of classical regression which is more distant when the M.F. becomes large se  \hyperref[figure11]{Figure 11}.
\begin{figure}[!h]
	\begin{minipage}[c]{.26\linewidth}
		\includegraphics[width=1.3\textwidth]{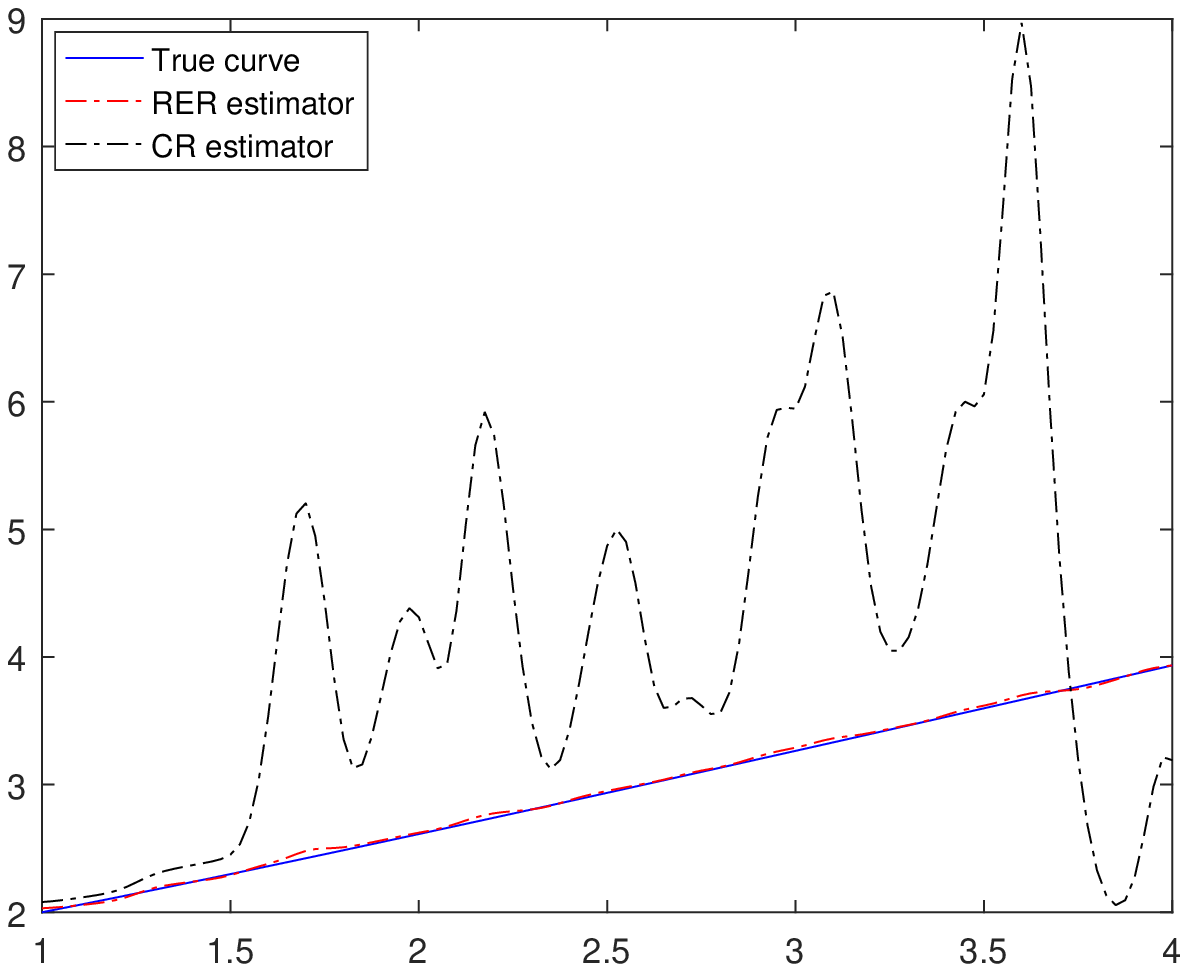}   
	\end{minipage} \hfill
	\begin{minipage}[c]{.26\linewidth}
		\includegraphics[width=1.3\textwidth]{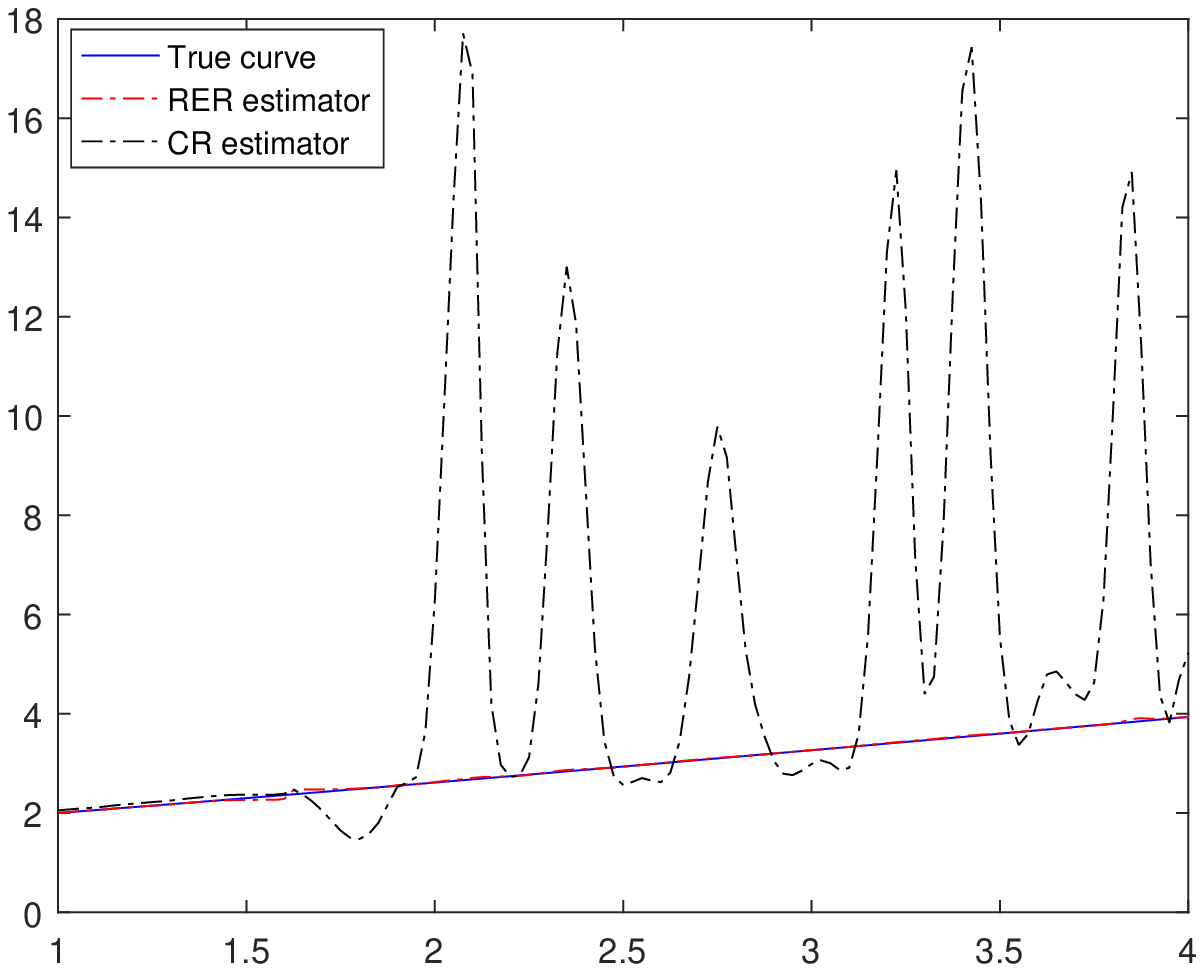}
	\end{minipage} \hfill
	\begin{minipage}[c]{.26\linewidth}
		\includegraphics[width=1.3\textwidth]{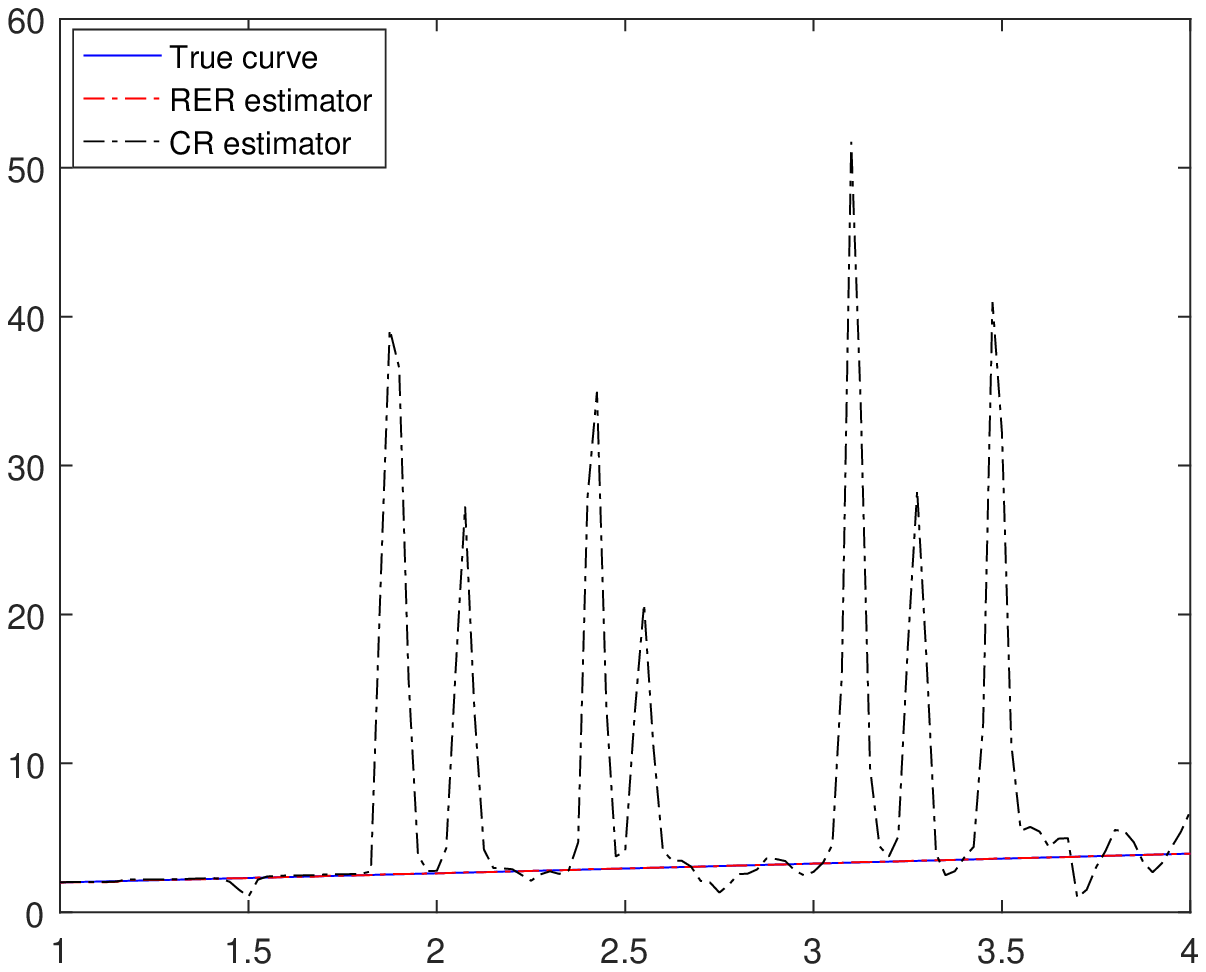}
	\end{minipage}\hfill\hfill
	\caption{\textcolor{blue}{$m(x)$}, \textcolor{red}{$\widehat{m}(x)$}, $\mu_n(x)$ with $c=1$, $\rho=0.7$, $n=300$, C.P. $\approx 35$ and M.F.=$10, 25 \;\text{and}\; 50\%$ respectively.}\label{figure10}
\end{figure}
\subsection{Discussion}
In this paper, an estimator for the relative error regression function on the multivariate case has been proposed, when the data are dependent and are subject to censoring.
After analyzing and comparing with the CR estimator, we have the following remarks. As expected, the asymptotic behavior of the RER estimator is better for a weak dependency (a small value of $\rho$) and a low censorship rate which is conforted by the numerical study in \hyperref[Sec4]{Sect. 4}, where we show how the quality of the estimation is influenced by several parameters (C.P., $\rho$, M.F., $n$). \\
Now, concerning the behavior of the RER estimator compared to the CR estimator, we can remark that the comportement of the RER remained almost unchanged in all our results in comparison with the CR estimator, which is significantly affected by the presence of outliers and censorship in the sample. Another interesting remark related to dependency is the fact that with small $\rho$ the estimator remains resistant. 
\section{Technical lemmas and proofs}\label{Sect5}
We split the proof of \hyperref[theorem]{Theorem 1} into following \hyperref[lem1]{Lemmatas 1-4}.
\begin{lem}\label{lem1}
	Under hypotheses \hyperref[K2]{\textbf{K2}} and \hyperref[D1]{\textbf{D1}}, for $\ell=1,2$,  we have
	\begin{equation*}
	\sup_{x \in \mathcal{C}}| \E [\tilde{r}_{\ell}(x)]- r_{\ell}(x)| =O_{a.s.}\left(h_n\right)  \quad as \quad n \longrightarrow \infty.
	\end{equation*}
\end{lem}
\begin{proof}
	The proof is standard in the sense that it is not affected by the dependency structure. Using the properties of conditional expectation, a change of variable, and Taylor's expansion $\zeta \in ]x-h_n t,x[$, we have under hypotheses \hyperref[K2]{\textbf{K2}},  \hyperref[D1]{\textbf{D1}} and for $\ell=1,2$
	\begin{equation*}
	\begin{aligned}
	\E [\tilde{r}_{\ell}(x)]- r_{\ell}(x)& =h_n^{-d}\int_{\R^d}K_d(x-u) m_{\ell}(u)f(u)du-r_{\ell}(x)\\
	& = h_n^{-d}\int_{\R^d}K_d(x-u) [r_{\ell}(u)-r_{\ell}(x)]du\\
	&=\int_{\R^d}K_d(t) [r_{\ell}(x-h_n t)-r_{\ell}(x)]dt\\
	\end{aligned}
	\end{equation*}
	then
	\begin{equation*}
	\begin{aligned}
	\sup_{x \in \mathcal{C}}| \E [\tilde{r}_{\ell}(x)]- r_{\ell}(x)]| &\leq h_n\sup_{x \in \mathcal{C}} \left|\int_{\R^d}K_d(t)\left(t_1 \frac{\partial r_{\ell}(\zeta)}{\partial x_1}+\dots+t_d \frac{\partial r_{\ell}(\zeta)}{\partial x_d}\right) dt\right|.
	\end{aligned}
	\end{equation*}
\end{proof}
\begin{lem} \label{lem2}
	Under hypotheses  \hyperref[H1]{\textbf{H1} and \hyperref[K1]{\textbf{K1-K3}}}, for $\ell=1,2$,  we have
	\begin{equation*}
	\sup_{x \in \mathcal{C}}|\widehat{r}_{\ell}(x)- \tilde{r}_{\ell}(x)| =O_{a.s.}\left(\left( \frac{log_2 n}{n}\right)^{1/2}\right)  \quad \text{as} \quad n \longrightarrow \infty.
	\end{equation*}
\end{lem}
\begin{proof}
	\begin{equation*}
	\begin{aligned}
	| \widehat{r}_{\ell}(x)- \tilde{r}_{\ell}(x)| &= \left| \frac{1}{n h_n^d} \sum_{i=1}^n \delta_i Y_i^{-\ell}K_d(x-X_i) \left( \frac{1}{\bar{G}_n(Y_i)}-\frac{1}{\bar{G}(Y_i)}\right)\right|\\
	&= \frac{1}{n h_n^d} \left|  \sum_{i=1}^n  T_i^{-\ell}K_d(x-X_i) \left( \frac{1}{\bar{G}_n(T_i)}-\frac{1}{\bar{G}(T_i)}\right)\right|\\
	&\leq \frac{\sup_{t \leq \tau_F} \left|\bar{G}_n(t)-\bar{G}(t)\right|}{ \bar{G}_n(\tau_F) \bar{G}(\tau_F)} \frac{1}{n h_n^d}\sum_{i=1}^n  |T_i|^{-\ell}K_d(x-X_i).
	\end{aligned}
	\end{equation*}
	From \hyperref[Cai1]{Cai (2001)}, under Hypotheses   \hyperref[H1]{\textbf{H1}} and \hyperref[K1]{\textbf{K1-K3}} , we have for $\ell=1,2$
	\[	\sup_{x \in \mathcal{C}}| \widehat{r}_{\ell}(x)- \tilde{r}_{\ell}(x)| \leq \frac{M}{\bar{G}^2(\tau_F)} \E[h_n^{-d} K_d(x-X_1)] \sqrt{\frac{\log \log n}{n}}.\]
\end{proof}
To this step, we introduce the following lemma (\hyperref[f_v]{Ferraty and Vieu (2006)} Proposition A.11 ii), p.237).
\begin{lem}[Fuk-Nagaev]\label{lem_FN}
	Let $\{U_i, i\geq 1\}$ be a sequence of real rv's, with strong mixing coefficient $\alpha(n)=O(n^{-v})$, $v>1$ such that $\forall n \in \mathbb{N}$, $\forall i \in \mathbb{N}$, $1 \leq i \leq n$ $|U_i|< +\infty$. Then for each $\varepsilon>0$ and for each $r>1$
	\[\mathbb{P}\left( \left|\sum_{i=1}^{n}U_i\right|>\varepsilon \right) \leq C\left(1+\frac{\varepsilon^2}{r S_n^2}\right)^{-r/2}+\frac{n C}{r}\left(\frac{2r}{\varepsilon}\right)^{v+1}\]
	where $\displaystyle{S_n^2=\sum_{i,j}|\C(U_i,U_j)|}$.
\end{lem}
In the following lemma we establish the asymptotic expression for the variance and covariance of the estimator $\widehat{m}(x)$.
\begin{lem} \label{lem4}
	Under hypotheses \hyperref[H2]{\textbf{H2}}, \hyperref[H3]{\textbf{H3}}, \hyperref[K1]{\textbf{K1-K3}} and \hyperref[D1]{\textbf{D1-D3}}, we have for $\ell=1,2$
	\begin{equation*}
	\sup_{x \in \mathcal{C}}|\tilde{r}_{\ell}(x)-\mathbb{E}[\tilde{r}_{\ell}(x)]| =O_{a.s.}\left(\sqrt{\frac{\log n}{nh_n^d}}+\sqrt{\frac{\log n}{n h_n^{2d/v}}}\right)  \quad as \quad n \longrightarrow \infty.
	\end{equation*}
\end{lem}	
\begin{proof}
Recall that $\mathcal{C}$ is a compact set, then it admits a covering $\mathcal{S}$ by a finite number $s_n$ of balls $\mathrm{B}_k(x^*_k,a_n^d)$ centred at $x^*_k=(x^*_{1,k},\dots,x^*_{d,k}), 1 \leq k \leq s_n $. Then for all $x \in \mathcal{C}$ there exists $ k $ such that $\norm{x-x^*_k}\leq a_n^d$ where $a_n$ verifies $a_n^{d\gamma}=h_n^{d\left(\gamma +\frac{1}{2}\right)}n^{-\frac{1}{2}}$ with $\gamma$ is the Lipshitz condition in hypothesis \hyperref[K1]{\textbf{K1}}. Since $\mathcal{C}$ is bounded then there exist a constant $M>0$ such that $s_n \leq M a_n^{-d}$.\\
	Let for $x \in \mathcal{C}$ and $\ell=1,2$ the given set
	\begin{equation*}
	\mathcal{A}_{\ell,i}(x)=(nh_n^d)^{-1} \left[\frac{\delta_i Y_i^{-\ell}}{\bar{G}(Y_i)}K_d(x-X_i)-\mathbb{E}\left(\frac{\delta_i Y_i^{-\ell}}{\bar{G}(Y_i)}K_d(x-X_i)\right)\right],
	\end{equation*}
	then
	\[\sum_{i=1}^{n}\mathcal{A}_{\ell,i}(x)=\tilde{r}_{\ell}(x)-\mathbb{E}[\tilde{r}_{\ell}(x)],\]
	that we decompose as follows
	\begin{equation*}
	\begin{aligned}
	\sum_{i=1}^{n}\mathcal{A}_{\ell,i}(x)&=\left\{ \left(\tilde{r}_{\ell}(x)-\tilde{r}_{\ell}(x^*_k)\right)-\left( \E\left[\tilde{r}_{\ell}(x)\right]- \E\left[\tilde{r}_{\ell}(x^*_k)\right] \right)\right\} + \left(\tilde{r}_{\ell}(x^*_k)-\E\left[\tilde{r}_{\ell}(x^*_k)\right]\right) \\
	&:=\sum_{i=1}^{n}\tilde{\mathcal{A}}_{\ell,i}(x)+\sum_{i=1}^{n}\mathcal{A}_{\ell,i}(x^*_k),
	\end{aligned}    
	\end{equation*}
	from where 
	\begin{equation*}
	\begin{aligned}
	\sup_{x \in \mathcal{C}} \left|\sum_{i=1}^{n}\mathcal{A}_{\ell,i}(x)\right| &\leq \max_{1 \leq k \leq s_n} \sup_{x \in B_k}\left|\sum_{i=1}^{n}\tilde{\mathcal{A}}_{\ell,i}(x)\right|+\max_{1 \leq k \leq s_n}\left|\sum_{i=1}^{n}\mathcal{A}_{\ell,i}(x^*_k)\right|\\
	&:=\mathcal{B}_1+\mathcal{B}_2.
	\end{aligned}    
	\end{equation*}
	We start by treating the first term $\mathcal{B}_1$
	\begin{equation*}
	\begin{aligned}
	\left|\sum_{i=1}^{n}\tilde{\mathcal{A}}_{\ell,i}(x)\right|&=\left| \left[\tilde{r}_{\ell}(x)-\tilde{r}_{\ell}(x^*_k)\right]- \E\left[\tilde{r}_{\ell}(x)-\tilde{r}_{\ell}(x^*_k)\right]\right|\\
	&=\frac{1}{nh_n^d}\sum_{i=1}^{n} \left| \frac{\delta_i Y_i^{-\ell}}{\bar{G}(Y_i)}\left(K_d(x-X_i)
	-K_d(x^*_k-X_i)\right)\right. \\
	&+ \left. \frac{1}{h_n^d}\E \left[\frac{\delta_1 Y_1^{-\ell}}{\bar{G}(Y_1)}\left(K_d(x-X_1) -K_d(x^*_k-X_1)\right)\right]\right|\\
	&\leq \frac{1}{nh_n^d} \sum_{i=1}^{n} \frac{ |T_i|^{-\ell}}{\bar{G}(T_i)}\left|K_d(x-X_i)-K_d(x^*_k-X_i) \right| \\
	&+ \frac{1}{h_n^d}\mathbb{E}\left[\frac{ |T_1|^{-\ell}}{\bar{G}(T_1)}\left|K_d(x-X_1)-K_d(x^*_k-X_1)\right|\right]\\
	&:=\mathcal{D}_{1,\ell}(x)+\mathcal{D}_{2,\ell}(x), 
	\end{aligned}
	\end{equation*}   
	with
	\begin{equation*}
	\begin{aligned}
	\sup_{x \in B_k}\mathcal{D}_{1,\ell}(x)& \leq \frac{ M}{\bar{G}(\tau_F)}\frac{1}{h_n^d}\sup_{x \in \mathcal{C}}\left|K_d(x-X_i)-K_d(x^*_k-X_i)\right|\\
	&\leq \frac{M}{h_n^d\bar{G}(\tau_F)}\norm{\frac{x-x^*_k}{h_n}}^{\gamma}\\
	& \leq \frac{C a_n^{d\gamma}}{h_n^{d+\gamma}}.
	\end{aligned}    
	\end{equation*}
	In the same manner, we have, 
	\[\sup_{x \in B_k}\mathcal{D}_{2,\ell}(x) \leq \frac{C a_n^{d\gamma}}{h_n^{d+\gamma}},\]
	then 
	\[\sup_{x \in B_k} \left|\sum_{i=1}^{n}\tilde{\mathcal{A}}_{\ell,i}(x)\right|=\sup_{x \in B_k}\mathcal{D}_{1,\ell}(x)+\sup_{x \in B_k}\mathcal{D}_{2,\ell}(x) \leq \frac{2 C a_n^{d\gamma}}{h_n^{d+\gamma}} \leq \frac{C h_n^{d(\gamma+\frac{1}{2})} n^{-\frac{1}{2}}}{h_n^{d+\gamma}}=\frac{C}{\sqrt{n h_n^{d}}} h_n^{\gamma(d-1)},\]
	which allows to 
	\begin{equation}
	\mathcal{B}_1=\max_{1 \leq k \leq s_n} \sup_{x \in B_k}\left|\sum_{i=1}^{n}\tilde{\mathcal{A}}_{\ell,i}(x)\right|=O\left(\frac{1}{\sqrt{n h_n^d}}\right).
	\end{equation}	
\end{proof}
\noindent To proceed to the determination of the second term $\mathcal{B}_2$, we will use \hyperref[lem_FN]{Lemma 3}.
Let
\[U_i=U_{i,k}=nh_n^d \mathcal{A}_{\ell,i}(x^*_k)=\frac{\delta_i Y_i^{-\ell}}{\bar{G}(Y_i)}K_d(x^*_k-X_i)-\mathbb{E}\left[\frac{\delta_i Y_i^{-\ell}}{\bar{G}(Y_i)}K_d(x^*_k-X_i)\right].\]
To apply \hyperref[lem_FN]{Lemme 4}, we have to calculate first
\begin{equation}\label{Sn}
\begin{aligned}
S_n^2&=\sum_i \sum_j|\C(U_i,U_j)|=\sum_{i\neq j}|\C(U_i,U_j)|+n Var (U_1)\\
&=: \mathcal{V}+n \V(U_1).
\end{aligned}
\end{equation}
On the one hand, we have to start by considering 
\begin{equation*}
\begin{aligned}
\V (U_1)&= Var \left[\frac{\delta_1 Y_1^{-\ell}}{\bar{G}(Y_1)}K_d(x^*_k-X_1)\right]\\
&=\E\left[\frac{\delta_1 Y_1^{-2 \ell}}{\bar{G}^2(Y_1)}K_d^2(x^*_k-X_1)\right]- \E^2\left[\frac{\delta_1 Y_1^{-\ell}}{\bar{G}(Y_1)}K_d(x^*_k-X_1)\right]\\
&=: \mathcal{R}_1-\mathcal{R}_2.
\end{aligned}    
\end{equation*}
For $\mathcal{R}_1$, using the conditional expectation propreties and a change of variables, we get 
\begin{equation*}
\mathcal{R}_1  \leq  h_n^d \int_{\R^d} K_d^2(t) \theta_{ \ell}(x^*_k-h_nt)dt,   
\end{equation*}
by a Taylor expansion and from Hypotheses \hyperref[D2]{\textbf{D2}} and \hyperref[K3]{\textbf{K3}}, we obtain
\begin{equation}\label{R1}
\mathcal{R}_1=O(h_n^d).
\end{equation}
For $\mathcal{R}_2$, under hypothesis \hyperref[D1]{\textbf{D1}}, we have
\begin{equation*}
\sqrt{\mathcal{R}_2} = \int_{\R^d} K_d(x^*_k-u) r_{ \ell}(u)du,  
\end{equation*}
using again a change of variable and a Taylor expansion around $x^*_k$, we have 
\begin{equation}\label{R2}
\mathcal{R}_2=O(h_n^{2d}).
\end{equation}
Then from (\ref{R1}) and (\ref{R2}), we get
\begin{equation} \label{nVar}
n \V(U_1)=n(\mathcal{R}_1-\mathcal{R}_2)=O(nh_n^d).
\end{equation}
On the other hand, 
\begin{equation*}
\begin{aligned}
|\C(U_i,U_j)|&=|\E[U_iU_j]|\\
&=\left|\E\left[\frac{\delta_i \delta_j Y_i^{-\ell}Y_j^{-\ell}}{\bar{G}(Y_i)\bar{G}(Y_j)} K_d(x^*_k-X_i) K_d(x^*_k-X_j) \right]\right.\\
&- \left.\E\left[\frac{\delta_i Y_i^{- \ell}}{\bar{G}(Y_i)}K_d(x^*_k-X_i) \right] \E\left[ \frac{\delta_j Y_j^{- \ell}}{\bar{G}(Y_j)}K_d(x^*_k-X_j)\right]\right|\\
&\leq h_n^{2d} \int_{\R^d}\int_{\R^d}K_d (t) K_d (s) |f_{i,j}(x^*_k-h_nt,x^*_k-h_ns)- f_{i}(x^*_k-h_n t) f_{j}(x^*_k-h_n s)| dt ds,
\end{aligned}    
\end{equation*}
which yeilds, under Hypothesis \hyperref[D3]{\textbf{D3}}
\begin{equation} \label{cov}
|\C(U_i,U_j)|=O(h_n^{2d}),
\end{equation}
uniformly on $i$ and $j$.\\
Now to evaluate the asymptotic behaviour of $\mathcal{V}$ following the decomposition of \hyperref[Masry]{Marsy (1986)}, we define the sets : 
\[E_1=\{ (i,j) \; \text{such that } \; 1 \leq|i-j|\leq \beta_n\}
\;\text{and}\; E_2=\{ (i,j) \; \text{such that } \; \beta_n+1 \leq|i-j|\leq n-1\}\]
where $\beta_n \rightarrow \infty$ as $n \rightarrow \infty $ at a slow rate, that is  $\beta_n = o(n)$. Let $\mathcal{V}_1$ and $\mathcal{V}_2$ be the sums of covariances over $E_1$ and $E_2$, respectively.
\[ \mathcal{V}= \sum_{i=1}^n \sum_{E_1} |\C (U_i,U_j|+ \sum_{i=1}^n \sum_{E_2} |\C (U_i,U_j| := \mathcal{V}_1+ \mathcal{V}_2.\]
We then get, from (\ref{cov})
\[\mathcal{V}_1= \sum_{i=1}^n \sum_{E_1} |\C (U_i,U_j|= \sum_{i=1}^n \sum_{1 \leq|i-j|\leq \beta_n} h_n^{2d}= O(n h_n^{2d}\beta_n).\]
For $\mathcal{V}_2$, we use the modified Davydov inequality for mixing processes (see \hyperref[Rio]{Rio (2000)}). This leads, for all $i \neq j$, to
\[|\C(U_i,U_j)| \leq C \alpha(|i-j|),\]
we then get,
\begin{equation*}
\begin{aligned}
\mathcal{V}_2 &\leq C \sum_{i=1}^n \sum_{\beta_n+1 \leq|i-j|\leq n-1} |i-j|^{-v}\\
&=O(n \beta_n^{(1-v)}).
\end{aligned}
\end{equation*}
Choosing $\beta_n = h_n^{-\frac{2d}{v}}$ permits to get,
\begin{equation}\label{Sn*}
\mathcal{V} = \mathcal{V}_1 + \mathcal{V}_2 = O(n h_n^{2d(v-1)/v}).
\end{equation}
Finally, from (\ref{Sn}),(\ref{Sn*}) and (\ref{nVar}) we obtain
\[S^2_n= O(n h_n^d)+O(n h_n^{2d(v-1)/v})=nh_n^d(1+h_n^{d(v-2)/v}).\]
Now, that all the calculus are done. It is convenient to apply the inequality in \hyperref[lem_FN]{Lemma 4} with $\varepsilon>0$
\begin{align*}
\mathbb{P} \left[\left|\sum_{i=1}^n \mathcal{A}_{\ell,i}(x^*_k)\right|> \varepsilon \right] &= \mathbb{P} \left[\left|\sum_{i=1}^n U_i\right|> nh_n^d \varepsilon \right]\nonumber\\ 
&\leq C \left(1+\frac{nh_n^d \varepsilon^2}{r (1+h_n^{d(v-2)/v})}\right)^{-\frac{r}{2}}+ nCr^{-1} \left( \frac{r}{nh_n^d \varepsilon}\right)^{v+1}\nonumber\\ 
&=:C(\mathcal{E}_1+\mathcal{E}_2).
\end{align*}
Taking $\varepsilon=\varepsilon_0 \left(\sqrt{\frac{\log n }{nh_n^d}}+\sqrt{\frac{\log n}{n h_n^{2d/v}}}\right)=:\varepsilon_n$ with $\varepsilon_0>0$, we get for the first part 
\begin{equation}\label{EE1}
\mathcal{E}_1=\left(1+C\frac{\varepsilon_0^2 \log n }{r}\right)^{-\frac{r}{2}}.
\end{equation}
By choosing $r=(\log n)^{1+b}$ with $b>0$, (\ref{EE1}) becomes 
\begin{equation*}
\mathcal{E}_1=\left(1+C\varepsilon_0^2 (\log n)^{-b}\right)^{-\frac{(\log n)^{1+b}}{2}}
\end{equation*}
by  taking logarithm and using a Taylor expansion of $\log(1+x)$ we get
\begin{equation*}\label{E1}
\log \mathcal{E}_1 \simeq (\log n)^{-\frac{C \varepsilon_0^2 }{2}}
\end{equation*}
which gives
\begin{equation} \label{Res_E1} \mathcal{E}_1=n^{-\frac{C \varepsilon_0^2 }{2}}. \end{equation}
For the same choice of $\varepsilon$ and $r$, we have 
\begin{equation}
\mathcal{E}_2 \simeq n (\log n)^{v(1+b)} \varepsilon_0^{-(v+1)} (nh_n^d  \log n )^{-\frac{v+1}{2}}.
\end{equation}
By taking again the inequality of Fuk-Nagaev and using $\mathbb{P}(\cup_i A_i)=\sum_i \mathbb{P}( A_i)$, we can write 
\begin{align}\label{R1-R2}
\mathbb{P} \left[\max_{1 \leq k \leq S_n}\left|\sum_{i=1}^n \mathcal{A}_{\ell,i}(x^*_k)\right|> \varepsilon_n \right]
&\leq M a_n^{-d} C \left(n^{-\frac{C \varepsilon_0^2 }{2}} + n (\log n)^{v(1+b)} \varepsilon_0^{-(v+1)} (nh_n^d  \log n )^{-\frac{v+1}{2}} \right) \nonumber \\
&\leq M  h_n^{-d(1+\frac{1}{2 \gamma})}  n^{\frac{1}{2 \gamma}} C \left(n^{-\frac{c \varepsilon_0^2 }{2}} + n (\log n)^{v(1+b)} \varepsilon_0^{-(v+1)} (nh_n^d  \log n )^{-\frac{v+1}{2}} \right) \nonumber \\
&\leq M C   n^{\frac{1}{2 \gamma}-\frac{C \varepsilon_0^2 }{2}} h_n^{-d(1+\frac{1}{2 \gamma})} \nonumber \\ 
&+ M C \varepsilon_0^{-(v+1)}   n^{1+\frac{1}{2 \gamma}} h_n^{-d(1+\frac{1}{2 \gamma})} (\log n)^{v(1+b)}  (nh_n^d  \log n )^{-\frac{v+1}{2}} \nonumber \\
&=: M C (\mathcal{R}_1+ \varepsilon_0^{-(v+1)} \mathcal{R}_2).
\end{align}
We have from hypothesis \hyperref[H3]{\textbf{H3}}
\begin{align}
\mathcal{R}_2 & \leq C n^{1+\frac{1}{2 \gamma}-\frac{\nu+1}{2}} h_n^{-d(1+\frac{1}{2 \gamma}+\frac{\nu+1}{2})} (\log n)^{\nu(1+b)-\frac{\nu+1}{2}} \nonumber \\
&\leq C n^{1+\frac{1}{2 \gamma}-\frac{\nu+1}{2}} n^{-\frac{(3-\nu)}{2}-\psi d \left[\frac{\gamma(\nu+1)+2\gamma+1}{2\gamma}\right]} (\log n)^{\nu(1+b)-\frac{\nu+1}{2}} \nonumber \\
&\leq C n^{-1+\frac{1-\psi d \left[\gamma(\nu+3)+1\right]}{2\gamma}} (\log n)^{\nu(1+b)-\frac{\nu+1}{2}}. \nonumber \\
\end{align}
Then, for an appropriate choice of $\psi$, $\mathcal{R}_2$ is the general term of a convergent series. In the same way, we can choose $\varepsilon_0$ such that $\mathcal{R}_1$ is the general term of convergent series. Finally, applying Borel-Cantelli's lemma to (\ref{R1-R2}) gives the result.
\begin{remark}
	The parameter $\psi$ of the hypothesis \hyperref[H3]{\textbf{H3}} can be chosen such as :
	\begin{equation}
	\psi > \frac{1}{\gamma(\nu+3)+1}.
	\end{equation}
\end{remark}
This condition ensures the convergence of the series of \hyperref[lem4]{Lemma 4}.
\begin{proof}\label{theo} {\it of} {\bf Theorem 1.}
	For $x \in \mathbb{R}^d$,  we consider the following decomposition : 
	\begin{multline*}
	\widehat{m}(x)-m(x) =\frac{1}{\widehat{r}_{2}(x)} \left\{ \Big[\big(\widehat{r}_{1}(x)-\tilde{r}_{1}(x)\big) + (\tilde{r}_{1}(x)-\mathbb{E}[\tilde{r}_{1}(x)])+ (\mathbb{E}[\tilde{r}_{1}(x)]-r_{1}(x))\Big] \right.\\ 
	+ r(x) \left. \Big[( \tilde{r}_{2}(x)-\widehat{r}_{2}(x))+( \mathbb{E}[\tilde{r}_{2}(x)]-\tilde{r}_{2}(x))+  (r_{2}(x)-\mathbb{E}[\tilde{r}_{2}(x)])\Big] \right\}
	\end{multline*}
	which by triangle inequality, we have
	\begin{multline}\label{ineq}
	\sup_{x \in \mathcal{C}} | \widehat{m}(x)-m(x) |\\
	\leq \frac{1}{\inf_{x \in \mathcal{C}}|\widehat{r}_{2}(x)|} \left\{\sup_{x \in \mathcal{C}} \Big[|\widehat{r}_{1}(x)-\tilde{r}_{1}(x)| + |\tilde{r}_{1}(x)-\mathbb{E}[\tilde{r}_{1}(x)]|+ | \mathbb{E}[\tilde{r}_{1}(x)]-r_{1}(x)|\Big] \right.\\ 
	+ \sup_{x \in \mathcal{C}}|r(x)| \left. \Big[|\tilde{r}_{2}(x)-\widehat{r}_{2}(x)|+| \mathbb{E}[\tilde{r}_{2}(x)]-\tilde{r}_{2}(x)|+ | r_{2}(x)-\mathbb{E}[\tilde{r}_{2}(x)]|\Big] \right\}.
	\end{multline}
	Then from the \hyperref[lem1]{Lemmas 1-4} in conjunction with the inequality (\ref{ineq}) conclude the proof.
\end{proof}


\begin{thebibliography}{spmpsci}
%
%
\bibitem{}\label{Attouch}
Attouch M, Laksaci A, Messabihi N (2015), {Nonparametric relative error regression for spacial random variables}, Stat Papers, 58:987-1008
\bibitem{}\label{Bollerslev}
Bollerslev T (1986), General autoregressive conditional heteroscedasticity, {J Econom }, {31}:307-327
\bibitem{}\label{Bosq}
Bosq, D (1998), {Nonparametric statistics for stochastics processes Estimation and Prediction}, {Lecture Notes in Statistics}, 110 Springer-Verlag New-York
\bibitem{}\label{bradley}
Bradley RD (2007), {Introduction to strong mixing conditions,} Vol I-III Kendrick Pres Utah
\bibitem{}\label{Cai1}
Cai Z (1998), {Asymptotic propreties of Kaplan-Meier estimator for censored dependent data}, {Stat Probab Lett,} 37: 381-389
\bibitem{}\label{Cai2}
Cai Z (2001), {Estimating a distribution function for censored time series data}, {J. Multiv. Analysis}, 78: 299-318 
\bibitem{}\label{Dabrowska}
Dabrowska M D (1987), { Nonparametric regression with censored survival time data,} {Scand. J. Statist,} 14: 181-197 
\bibitem{}\label{Demongeot}
Demongeot  J, Hamie A, Laksaci A, Rachdi M (2016), {Relative error prediction in nonparametric functional statistics: Theory and Practice}, {J. Multivariate Anal.}, 146: 261-268
\bibitem{}\label{Ghouch}
El Ghouch A, Van Keilegom I, (2008), {Nonparametric regression with dependent censored data}, {Scand. J. Statist.}, 35: 228-247
\bibitem{}\label{Engle} 
Engle, R F, (1982), Autoregressive conditional heteroskedasticity with estimates of the variance of U.K. inflation, {Econometrica}, 50:987-1007
\bibitem{}\label{f_v}
Ferraty F, Vieu P, (2006), {Non parametric functionnal data analysis,} {Theory and Practice,} Springer-Verlag, New-York
\bibitem{}\label{Guessoum1}
Guessoum Z, Ould Said E, (2008), {On the nonparametric estimation of the regression function under censorship model}, {Statist. and Decisions}, 26: 159-177
\bibitem{}\label{Guessoum}
Guessoum Z, Ould Said E, (2010), {Kernel regression uniform rate estimation for censored data under $\alpha$-mixing condition}. \emph{Elect. J. of statist.}, 4: 117-132.
\bibitem{}\label{jones} 
Jones D A, (1978), Nonlinear autoregressive processes, {Proc. Roy. Soc. London A}, 360, 71-95
\bibitem{}\label{Kap-Mei}
Kaplan E L, Meier P, (1958), {Nonparametric estimation for incomplete observations}, {J. Amer. Statist. Assoc.}, 53: 457-481
\bibitem{}\label{Li}
Li X, Yang W, Hu S, (2016), {Uniform convergence of estimator for nonparametric regression with dependent data,} {J. of inequal. and Appli.} 142
\bibitem{}\label{lipsitz}
Lipsitz S R, Ibrahim J G, (2000). {Estimation with correlated censored survival data with missing covariates,} {Biostatistics}, 19: 315-327
\bibitem{}\label{Makridakis}
Makridakis S, (1984), {The forecasting Accuracy of Major Time Series Methods}, \emph{Wiley,} New York, 1984
\bibitem{}\label{Masry}
Masry E, (1986), {Recursive probability density estimation for weakly dependent stationary processes}, {IEEE Trans. Inform. theory}, 32: 254–267
\bibitem{}\label{Narula}
Narula S C, Wellington, J F (1977), {Prediction, linear regression and the minimum sum of relative errors}, {Technometrics,} 19: 185-190
\bibitem{}\label{Ozaki}
Ozaki, T (1979), Nonlinear time series models for nonlinear random vibrations, Technical report. Univ. of Manchester
\bibitem{}\label{Park}
Park H, Stefanski L A, (1998), {Relative error prediction}, {Statistics and Probability Letters}, 40: 227-236
\bibitem{}\label{Jones}
Park H, Shin K I, Jones M C, Vines S K (2008), {Relative error prediction via kernel regression smoothers}, {Journal of Stat. Plann. and Infer.,} 138, 2887-2898
\bibitem{}\label{Rio}
Rio E, (2000) {Theorie asymptotique des processus aléatoires faiblement dépendants}, {Math.}, 42: 43–47
\bibitem{}\label{Rosenblatt}
Rosenblatt M, (1956), {A central limit theorem and a strong mixing condition}, {Proc. Nat. Acad. Sci. USA}, 42: 43–47
\bibitem{}\label{Shen}  
Shen J, Xie Y,(2013), {Strong consistency of the internal estimator of the nonparametric regression with dependent data,} {Statistics and Probability Letters,} 83, 1915-1925
\end{thebibliography}


\end{document}